\documentclass{amsart}

\usepackage{amsmath, amscd}
\usepackage{amssymb}
\usepackage[alphabetic]{amsrefs}
\usepackage[T1]{fontenc}
\usepackage[all]{xy}
\usepackage{mathtools}
\usepackage{footmisc}
\usepackage{stmaryrd}
\usepackage{mathrsfs}

\usepackage{tikz}
\usetikzlibrary{calc}
\usetikzlibrary{matrix,arrows,decorations.pathmorphing}
\usepackage{tikz-cd}

\usepackage{relsize} 
\usepackage[bbgreekl]{mathbbol} 
\usepackage{amsfonts} 
\usepackage[colorlinks=true, hyperindex, linkcolor=cyan, citecolor=, pagebackref=false, urlcolor=cyan, linktocpage]{hyperref}

\DeclareSymbolFontAlphabet{\mathbb}{AMSb} 
\DeclareSymbolFontAlphabet{\mathbbl}{bbold}

\theoremstyle{definition}
\newtheorem{thm}{Theorem}[section]
\newtheorem{lem}[thm]{Lemma}
\newtheorem{cor}[thm]{Corollary}
\newtheorem{prop}[thm]{Proposition}

\theoremstyle{definition}

\newtheorem{rem}[thm]{Remark}

\newtheorem{dfn}[thm]{Definition}

\newtheorem{ex}[thm]{Example}
\newtheorem{construction}[thm]{Construction}
\newcommand{\bA}{\mathbb{A}}

\newcommand{\bG}{\mathbb{G}}

\newcommand{\bN}{\mathbb{N}}
\newcommand{\bZ}{\mathbb{Z}}
\newcommand{\bF}{\mathbb{F}}

\newcommand{\cA}{\mathcal{A}}

\newcommand{\cE}{\mathcal{E}}
\newcommand{\cF}{\mathcal{F}}
\newcommand{\cH}{\mathcal{H}}

\newcommand{\cM}{\mathcal{M}}
\newcommand{\cN}{\mathcal{N}}
\newcommand{\cO}{\mathcal{O}}
\newcommand{\cP}{\mathcal{P}}

\newcommand{\sC}{\mathscr{C}}

\newcommand{\BT}{\mathrm{BT}}

\newcommand{\et}{\text{\'{e}t}}
\newcommand{\Ext}{\mathrm{Ext}}

\newcommand{\fin}{\mathrm{fin}}

\newcommand{\fs}{\mathrm{fs}}
\newcommand{\Hom}{\mathrm{Hom}}

\newcommand{\Ker}{\mathrm{Ker}}
\newcommand{\kfl}{\mathrm{kfl}}

\newcommand{\Spec}{\mathrm{Spec}}

\newcommand{\rk}{\mathrm{rk}}

\makeatletter
    
    \@addtoreset{equation}{section}
  \makeatother

\begin{document}

\title[Slope filtrations of log $p$-divisible groups]{Slope filtrations of log $p$-divisible groups}

\author{Kentaro Inoue}
\address{Department of Mathematics, Faculty of Science, Kyoto University, Kyoto 606-8502, Japan}
\email{keninoue0123@gmail.com}

\begin{abstract}
Oort-Zink proved that a $p$-divisible group over a normal base in characteristic $p$ with constant Newton polygon is isogenous to a $p$-divisible group admitting a slope filtration. In this paper, we generalize this result to log $p$-divisible groups. 
\end{abstract}

\maketitle

\section{Introduction}

Zink proved that a $p$-divisible group over a regular $\bF_p$-scheme with constant Newton polygon is isogenous to a $p$-divisible group admitting a slope filtration (\cite{zin01}), and Oort and Zink generalized this result to the case of a normal base (\cite{oz02}). First, we briefly recall the result of Oort-Zink. Let $S$ be a locally noetherian normal $\bF_p$-scheme. Let $G$ be a $p$-divisible group over $S$ with constant Newton polygon. Then there exists a completely slope divisible $p$-divisible group $H$ and an isogeny $G\to H$. Here a $p$-divisible group $H$ over $S$ is called \emph{completely slope divisible} if $H$ admits a filtration of $p$-divisible subgroups
  \[
  0=H_0\subset H_1 \subset \dots \subset H_m=H,
  \]
  and if there are integers
  \[
  s\geq r_1> r_2 >\dots >r_m\geq 0
  \]
  such that the following conditions are satisfied:
  \begin{enumerate}
      \item The quasi-isogeny $p^{-r_i}F^s\colon H_i\to H_{i}^{(p^s)}$ is an isogeny for $i=1, \dots ,m$.
      \item The $p$-divisible group $H_{i}/H_{i-1}$ is isoclinic of slope $r_i/s$ for $i=1, \dots ,m$.
  \end{enumerate}
The above theorem plays an important role in the theory of foliations in the moduli space of abelian varieties in characteristic $p$ (see \cite{oor04}).

In this paper, we shall prove a log version of this theorem. \emph{Log $p$-divisible groups} are degenerating objects of $p$-divisible groups in the framework of logarithmic geometry. In \cite{kkn08} and \cite{kkn21}, Kajiwara-Kato-Nakayama introduced the notion of \emph{log abelian varieties}, which are degenerating objects of abelian varieties in the framework of logarithmic geometry, and realized the toroidal compactification of the moduli space of abelian varieties constructed by Faltings-Chai as the fine moduli of log abelian varieties. In the same way as non-log case, the $p$-torsion part of a log abelian variety is a log $p$-divisible group. We expect that the study in this paper is helpful to study the structure of the toroidal compactification of the moduli space of abelian varieties in characteristic $p$.

The following is one of the main results of this paper.

\begin{thm}[see Corollary \ref{first main thm}]\label{first main intro}
    Let $S$ be a normal fs log scheme over $\bF_{p}$. Then, for a log $p$-divisible group $G$ over $S$ with constant Newton polygon, there exist a completely slope divisible log $p$-divisible group $H$ and an isogeny $G\to H$.
\end{thm}

This result can be deduced from the corresponding result of Oort-Zink by using the fact that a completely slope divisible log $p$-divisible group is written as an extension of a classical \'{e}tale $p$-divisible group by a classical completely slope divisible formal Lie group (see Proposition \ref{dec of csd}). 

Next, we consider a stronger statement. Oort-Zink proved the above result by extending an isogeny generically defined. We prove that we can make the same procedure for log $p$-divisible groups.

\begin{thm} [see Theorem \ref{ext of isog over log regular base}]\label{main intro}
    Let $S$ be a log regular fs log scheme over $\bF_p$ and $U$ be a dense open subscheme of $S$. Let $G$ be a log $p$-divisible group over $S$. Suppose that we are given a completely slope divisible log $p$-divisible group $H_{U}$ over $U$ and an isogeny $G|_{U}\to H_{U}$. Then there exist a completely slope divisible log $p$-divisible group $H$ over $S$ and an isogeny $G\to H$ which restricts to the given isogeny $G|_{U}\to H_{U}$.
\end{thm}

When $\mathring{S}$ is regular, this theorem follows from a result on the extendability of complete slope divisible log $p$-divisible groups (see Proposition \ref{purity of csd}). In general cases, we reduce the problem to the above case by using resolutions of toric singularities in \cite{niz06}.

Note that we need the assumption of log regularity. Theorem \ref{main intro} does not hold under the assumption of Theorem \ref{first main intro} (for a counterexample, see Remark \ref{need log reg}).

The outline of this paper is as follows. In Section \ref{sec 2}, we review the definitions of Kummer log flat topology, log vector bundles, log finite group schemes, and log $p$-divisible groups, and we prove some fundamental properties of them. In Section \ref{sec 3}, we introduce the notion of completely slope divisible log $p$-divisible groups and Newton polygons for log $p$-divisible groups. In Section \ref{sec 4}, we review the definition and properties of log regular fs log scheme, and we study log finite group schemes and log vector bundles on log regular fs log schemes for the proof of Theorem \ref{ext of isog over log regular base}. In Section \ref{sec 5}, we shall prove Theorem \ref{ext of isog over log regular base}.

\subsection*{Notation and convention}\noindent

\begin{enumerate}
    \item Throughout this paper, $p$ is a prime number.
    \item All monoids are commutative. For a monoid $M$, let $\overline{M}\coloneqq M/M^{\times}$, and let $M^{\rm{gp}}$ be an abelian group associated to $M$. We say that $M$ is \emph{sharp} if $M^{\times}=1$.
    \item For an integer $n\geq 1$, a monoid $M$ equipped with a monoid map $\times n\colon M\to M$ is denoted by $M^{1/n}$.
    \item For a log scheme $S$, the structure sheaf of monoids is denoted by $\cM_{S}$. The underlying scheme of $S$ is denoted by $\mathring{S}$. 
    \item Let $\cP$ be a property of schemes, including quasi-compact, quasi-separated, finite, and locally noetherian. We say that a log scheme $S$ satisfies $\cP$ if the underlying scheme $\mathring{S}$ satisfies $\cP$. 
    \item Let $\cP$ be a property of morphisms of schemes. We say that a morphism of log schemes $T\to S$ satisfies $\cP$ if the morphism of the underlying schemes $\mathring{T}\to \mathring{S}$ satisfies $\cP$.  
    \item For a monoid $M$, the log scheme $\mathrm{Spec}\bZ[M]$ equipped with the log structure associated with a natural monoid map $M\to \cO_{\mathrm{Spec}\bZ[M]}$ is denoted by $\bA_{M}$.
    \item A chart $P\to \cM_{S}$ is called \emph{neat} at $s\in S$ if $P\to \overline{\cM_{S, \bar{s}}}$ is an isomorphism. For an fs log scheme $S$ and a point $s\in S$, there exists an fs chart $P\to \cM_{S}$ which is neat at $s$ by \cite[Chapter II, Proposition 2.3.7]{ogu18}
    \item Unless otherwise specified, fiber products of saturated log schemes are taken in the category of saturated log schemes. For morphisms of saturated log schemes $S_{1}\to S_{3}$ and $S_{2}\to S_{3}$, the saturated fiber product of $S_{1}$ and $S_{2}$ over $S_{3}$ is simply denoted by $S_{1}\times_{S_{3}} S_{2}$
\end{enumerate}

\section{Preliminaries on log schemes and log \texorpdfstring{$p$}--divisible groups}\label{sec 2}

Throughout this section, let $S$ denote an fs log scheme.

\subsection{Kummer log flat topologies} \noindent

Let $(\fs/S)$ denote the category of fs log schemes over $S$. Note the category of schemes over $\mathring{S}$ is regarded as a full subcategory of $(\fs/S)$ by equipping schemes over $\mathring{S}$ with the pullback log structure of $\cM_{S}$.

In this paper, we use Kummer log flat topology introduced by Kato; see \cite[Section 2]{kat21}. Recall that, for $U\in (\fs/S)$, a family of morphisms $\{f_{i}\colon U_{i}\to U\}$ is called a \emph{Kummer log flat cover} if the following conditions are satisfied.
    \begin{enumerate}
        \item Each $f_i$ is log flat, of Kummer type, and locally of finite presentation.
        \item The family is set-theoretically surjective, i.e. $\mathring{U}=\cup f_{i}(\mathring{U_{i}})$.
    \end{enumerate}

The \emph{Kummer log flat topology} on $(\fs/S)$ is the Grothendieck topology given by Kummer log flat covers on $(\fs/S)$. The resulting site is denoted by $(\fs/S)_{\kfl}$.

\begin{prop}
    For an fs log scheme $T$ over $S$, the representable presheaf by $T$ on $(\fs/S)$ is a sheaf.
\end{prop}

\begin{proof}
    See \cite[Theorem 3.1]{kat21}.
\end{proof}

The following lemma is a useful property of Kummer morphisms.

\begin{lem} \label{kum}
    Let $f\colon T\to S$ be a Kummer morphism of fs log schemes. Suppose that $\mathring{T}$ is quasi-compact, and that we are given an fs chart $P\to \cM_{S}$. Then there exist an integer $n\geq 1$ such that, if we put $S'=S\times_{\bA_{P}} \bA_{P^{1/n}}$, the natural morphism $T\times_{S} S'\to S'$ is strict. 
\end{lem}

\begin{proof}
    See the proof of \cite[Proposition 2.7 (2)]{kat21}.
\end{proof}

\subsection{Log vector bundles} \noindent 

Let $\cO_{S_{\kfl}}$ be a sheaf of rings on $(\fs/S)_{\kfl}$ given by $T\mapsto \Gamma(\mathring{T}, \cO_{\mathring{T}})$. The fact that $\cO_{S_{\kfl}}$ is a sheaf on $(\fs/S)_{\kfl}$ is proved in \cite[\S 3.4]{kat21}. For a quasi-coherent sheaf $\cF$ on $\mathring{S}$, we consider the $\cO_{S_{\kfl}}$-module $\alpha \cF$ given by $T\mapsto \Gamma(\mathring{T},f^{*}\cF)$, where $f\colon \mathring{T}\to \mathring{S}$ is the structure morphism (cf.~\cite[Proposition 2.19]{niz08}). This gives a fully faithful functor $\alpha$ from the category of quasi-coherent $\cO_{\mathring{S}}$-modules to the category of $\cO_{S_{\kfl}}$-modules. The functor $\alpha$ is right exact and preserves exact sequences of locally free $\cO_{X}$-modules. 

\begin{dfn}[Log vector bundles] 
For a scheme $X$, we let $\mathcal{LF}(X)$ denote the category of vector bundles on $X$ (i.e.~locally free $\cO_{X}$-modules of finite rank).

A \emph{log vector bundle} on $S$ is a locally free sheaf of finite rank on the ringed site $((\fs/S)_{\kfl},\cO_{S_{\mathrm{kfl}}})$. In other words, an $\cO_{S_{\kfl}}$-module $\cE$ is called a log vector bundle if there exists a Kummer log flat cover $\{U_i\to S\}$ such that each restriction $\cE|_{U_i}$ is isomorphic to $\alpha \cF_{i}$ for a vector bundle $\cF_{i}$ on $\mathring{U_{i}}$. A log vector bundle $\cE$ is called \emph{classical} if $\cE$ is isomorphic to $\alpha \cF$ for a vector bundle $\cF$ on $\mathring{S}$. The category of log vector bundles on $S$ is denoted by $\mathcal{LLF}(S)$. The functor $\alpha$ induces a fully faithful functor
\[
\mathcal{LF}(\mathring{S})\to \mathcal{LLF}(S),
\]
which is also denoted by $\alpha$. Obviously, the functor $\alpha$ gives an equivalence from the category $\mathcal{LF}(\mathring{S})$ to the category of classical log vector bundles on $S$, and we identify these categories.
\end{dfn}

\begin{lem}\label{log qcoh cl}
    Let $S$ be a quasi-compact fs log scheme with an fs chart $P\to \cM_{S}$. Let $\cE$ be a log vector bundle on $S$. Then there exists an integer $n\geq 1$ such that, if we set $S'\coloneqq S\times_{\bA_{P}} \bA_{P^{1/n}}$, the pullback of $\cE$ by $S'\to S$ is classical.
\end{lem}

\begin{proof}
    Take a Kummer log flat cover $T\to S$ such that $\cE|_{T}$ is classical. We may assume that $\mathring{T}$ is quasi-compact. By Lemma \ref{kum}, there exist an integer $n\geq 1$ such that, if we set
    \[
    S'\coloneqq S\times_{\bA_{P}} \bA_{P^{1/n}},\ \  T'\coloneqq T\times_{S} S',
    \]
    the natural morphism $T'\to S'$ is a strict fppf cover. As $\cE|_{T}$ is classical, $\cE|_{T'}$ is also classical. By fppf descent, $\cE|_{S'}$ is also classical.
\end{proof}

\begin{prop}\label{lim log loc free}
    Let $\{ S_{i}\}_{i\in I}$ be a cofiltered system of fs log schemes satisfying the condition $(\ast)$ at the beginning of Appendix \ref{appendix}. We put $S\coloneqq \varprojlim_{i\in I} S_i$. 
    \begin{enumerate}
        \item Suppose that the cofiltered category $I$ has a final object $0$. Let $\cE_{0}$ and $\cF_{0}$ be objects of $\mathcal{LLF}(S_{0})$. We put $\cE_{i}\coloneqq \cE_{0}|_{S_{i}}$, $\cF_{i}\coloneqq \cF_{0}|_{S_{i}}$, $\cE\coloneqq \cE_{0}|_{S}$, and $\cF\coloneqq \cF_{0}|_{S}$. Then the natural map
        \[
        \displaystyle \varinjlim_{i\in I}\Hom(\cE_{i}, \cF_{i})\to \Hom(\cE,\cF)
        \]
        is an isomorphism.
        \item For any $\cE\in \mathcal{LLF}(S)$, there exist $i\in I$ and $\cE_{i}\in \mathcal{LLF}(S_{i})$ with an isomorphism $\cE_{i}|_{S}\cong \cE$.
    \end{enumerate}
\end{prop}

\begin{proof} 
    (1) By working Kummer log flat locally on $S_{0}$, this statement follows from the limit argument for usual vector bundles on schemes.

    (2) Take a qcqs Kummer log flat cover $T\to S$ such that $\cE|_{T}$ is classical. By Proposition \ref{lim log sch} (2) and Proposition \ref{lim log properties} (4), $T\to S$ comes from a Kummer log flat cover $T_{i}\to S_{i}$ and $\cE|_{T}$ comes from a vector bundle $\cE_{T_{i}}$ over $\mathring{T_{i}}$ for some $i\in I$. By replacing $i$ with a bigger object of $I$, we may assume that the descent datum of $\cE|_{T}$ over $T\times_{S} T$ comes from the descent datum of $\cE_{T_{i}}$ over $T_{i}\times_{S_{i}} T_{i}$, which defines an object $\cE_{i}$ of $\mathcal{LLF}(S_{i})$ with an isomorphism $\cE_{i}|_{S}\cong \cE$.
\end{proof}

\begin{prop}[cf.~{\cite[Theorem 6.2]{kat21}}]\label{pointwise kfl vect bdle}
    Let $\cE\in \mathcal{LLF}(S)$. Suppose that the pullback of $\cE$ to a point $s\in S$ is classical. Then there exists an open neighborhood $U\subset S$ of $s$ such that $\cE|_{U}$ is classical. In particular, if the pullback of $\cE$ to every point $s\in S$ is classical, then $\cE$ is classical.
\end{prop}

\begin{proof}
    By the limit argument (Proposition \ref{lim log loc free}), we may assume that $\mathring{S}$ is the spectrum of a noetherian strict local ring. In this case, the assertion follows from  \cite[Theorem 6.2]{kat21}.
\end{proof}

\subsection{Log finite group schemes} \noindent

In this subsection, we review the definition of log finite group schemes introduced by Kato in \cite{kat23} and prove some fundamental properties. For a scheme $X$, the category of finite locally free group schemes over $X$ is denoted by $(\mathrm{fin}/X)$.

\begin{dfn}[\cite{kat23}, Section 1] \noindent

(1) Let $(\fin/S)_{\rm{c}} $ be the category of sheaves of abelian groups on $(\fs/S)_{\kfl} $ represented by finite locally free group schemes over $\mathring{S}$ which are considered as strict log schemes over $S$. Obviously, the category $(\fin/S)_{\rm{c}}$ is naturally equivalent to the category $(\mathrm{fin}/\mathring{S})$, and we identify these categories.

(2) Let $(\fin/S)_{\rm{f}} $ be the category of sheaves $G$ of abelian groups on $(\fs/S)_{\kfl} $ such that there exists a Kummer log flat cover $\{U_{i}\to S \} $ with $G|_{U_{i}}\in (\fin/U_{i})_{\rm{c}} $. For $G\in (\fin/S)_{\rm{f}}$, the sheaf $G^{*}\coloneqq \cH\! \mathit{om}(G, \bG_{m})$ on $(\fs/S)_{\kfl} $ is called the \emph{Cartier dual} of $G$. When $G$ belongs to $(\fin/S)_{\rm{c}}$, the sheaf $G^{*}$ is representable by the Cartier dual of $G$ in the usual sense. By working Kummer log flat locally on $S$, we see that $G^{*}$ belongs to $(\fin/S)_{\rm{f}}$ for $G\in (\fin/S)_{\rm{f}}$.

(3) Let $(\fin/S)_{\rm{d}} $ be the full subcategory of $(\fin/S)_{\rm{f}} $ consisting of objects $G$ such that both of $G$ and $G^{*}$ are representable by finite Kummer log flat fs log schemes of finite presentation over $S$. Objects in $(\fin/S)_{\rm{d}} $ are called \emph{log finite group schemes} over $S$.

We have the inclusion relation of categories
    \[
    (\fin/S)_{\rm{c}}\subset (\fin/S)_{\rm{d}}\subset (\fin/S)_{\rm{f}}.
    \]
When an object $G\in (\fin/S)_{\rm{f}}$ belongs to $(\fin/S)_{\rm{c}}$, we say that $G$ is \emph{classical}.
\end{dfn}

\begin{dfn}
    For $G\in (\fin/S)_{\rm{f}}$, we say that $G$ is of \emph{order $n$} if there exists a Kummer log flat cover $\{U_{i}\to S\} $ such that, for each $i$, $G|_{U_{i}}$ is a finite locally free group scheme of order $n$ over $\mathring{U_{i}}$. 
\end{dfn}

\begin{ex}
    If the log structure of $S$ is trivial, the category $(\fin/S)_{\rm{f}} $ is equal to $(\fin/S)_{\rm{c}} $. This follows from the fact that, for any fs log scheme $T$ with a  Kummer morphism $T\to S$, the log structure of $T$ is also trivial.
\end{ex}

\begin{ex}\label{log av torsion}
    Let $n\geq 1$ be an integer. Every log abelian variety $A$ over $S$ (for its definition, see \cite[Definition 4.1]{kkn08}) is a sheaf on $(\fs/S)_{\kfl}$ by \cite[Theorem 2.1 (1)]{zha17}. The sheaf $A[n]:=\Ker(\times n\colon A\to A)$ is a log finite group scheme of order $n^{2\dim(A)} $ over $S$ by \cite[Proposition 18.1 (2)]{kkn15}, Proposition \ref{fiberwise} below, and \cite[Proposition 3.4 (3)]{zha17}.
\end{ex}

\begin{ex}\label{log fin grp ass to q}
    Let $S$ be an fs log scheme. Let $\bG_{m, \log}$ denote a sheaf on $(\fs/S)_{\kfl}$ given by $\bG_{m, \log}(T)=\Gamma(\mathring{T}, \cM_{T}^{\rm{gp}})$ for $T\in (\fs/S)$. The fact that $\bG_{m, \log}$ is a sheaf is proved in \cite[Theorem 3.2]{kat21}. For an integer $n\geq 1$, there is a Kummer sequence (\cite[Proposition 4.2]{kat21})
    \[
    0 \to \mu_{n} \to \bG_{m, \log} \stackrel{(-)^{n}}{\to} \bG_{m, \log} \to 0.
    \]
    From this sequence, we get a connecting homomorphism 
    \[
    \bG_{m, \log}(S)=\Gamma(\mathring{S}, \cM_{S}^{\rm{gp}})\to H^{1}((\fs/S)_{\kfl}, \mu_{n})=\Ext^{1}_{(\fs/S)_{\kfl}}(\bZ/n\bZ, \mu_{n})
    \]
    which associates to any $q\in \Gamma(S, \cM_{S}^{\rm{gp}})$ an extension of sheaves on $(\fs/S)_{\kfl}$
    \[
    0\to \mu_{n}\to G_{q}\to \bZ/n\bZ\to 0.
    \]
    The sheaf $G_{q}$ is a log finite group scheme over $S$ by \cite[Proposition 2.3]{kat23}.
\end{ex}

\begin{lem}\label{kfl descent for weak log fin}
    The category $(\mathrm{fin}/S)_{\mathrm{f}}$ satisfies Kummer log flat descent with respect to $S$.
\end{lem}

\begin{proof}
    For a site $\sC$, write $\mathrm{Shv}(\sC)$ for the category of abelian groups on $\sC$. Obviously, the category $\mathrm{Shv}((\mathrm{fs}/S)_{\mathrm{kfl}})$ satisfies Kummer log flat descent with respect to $S$. For $G\in \mathrm{Shv}((\mathrm{fs}/S)_{\mathrm{kfl}})$, the property that $G$ belongs to $(\mathrm{fin}/S)_{\mathrm{f}}$ is Kummer log flat locally on $S$. Therefore, the category $(\mathrm{fin}/S)_{\mathrm{f}}$ also satisfies Kummer log flat descent with respect to $S$.
\end{proof}

\begin{lem} \label{log fin grp cl}
Let $S$ be a quasi-compact fs log scheme with an fs chart $P\to \cM_{S}$. Let $G\in (\fin/S)_{\rm{f}}$. Then there exists an integer $n\geq 1$ such that, if we set $S'\coloneqq S\times_{\bA_{P}} \bA_{P^{1/n}}$, the pullback of $G$ by $S'\to S$ is classical.
\end{lem}

\begin{proof}
    This can be proved in the same way as Lemma \ref{log qcoh cl}.
\end{proof}

\begin{prop}\label{lim log fin grp sch}
Let $\{ S_{i}\}_{i\in I}$ be a cofiltered system of fs log schemes satisfying the condition $(\ast)$ at the beginning of Appendix \ref{appendix}. We put $S\coloneqq  \varprojlim_{i\in I} S_i$. 

(1) Suppose that the cofiltered category $I$ has a final object $0$. Let $G_{0}$ and $H_{0}$ be objects of $(\fin/S_{0})_{\rm{f}}$. We put $G_{i}:=G_{0}|_{S_{i}}$, $H_{i}:=H_{0}|_{S_{i}}$, $G:=G_{0}|_{S}$, and $H:=H_{0}|_{S}$. Then the natural map
        \[
        \displaystyle \varinjlim_{i\in I}\Hom(G_{i}, H_{i})\to \Hom(G, H)
        \]
        is an isomorphism.
        
(2) For any $G\in (\fin/S)_{\rm{f}}$, there exist $i\in I$ and $G_{i}\in (\fin/S_{i})_{\rm{f}}$ with an isomorphism $G_{i}|_{S}\cong G$.

(3) For any $G\in (\fin/S)_{\rm{d}}$, there exist $i\in I$ and $G_{i}\in (\fin/S_{i})_{\rm{d}}$ with an isomorphism $G_{i}|_{S}\cong G$.
\end{prop}

\begin{proof} 
    (1) By working Kummer log flat locally on $S_{0}$, this statement follows from the limit argument for usual finite locally free group schemes over schemes.

    (2) This follows from the same argument as Proposition \ref{lim log loc free} (2).
    
    (3) This follows from (1) and the limit argument for finite Kummer log flat fs log schemes of finite presentation (Proposition \ref{lim log sch} and Proposition \ref{lim log properties} (4)).
\end{proof}

\begin{rem}
    The theory of log finite group schemes is introduced and developed in \cite{kat23} under the assumption that the base log scheme is noetherian. Since an fs affine noetherian log scheme can be written as a limit of a cofiltered system of fs affine log schemes satisfying the condition ($\ast$) at the beginning of Appendix \ref{appendix} by Lemma \ref{lim log str}, Proposition \ref{lim log fin grp sch} enables us to generalize several results in \textit{loc. cit.} to the non-noetherian case.
\end{rem}

Let $G\in (\fin/S)_{\rm{f}}$. We let $\cA_{G}$ denote the sheaf of morphisms $\mathcal{M}or(G,\cO_{S_{\mathrm{kfl}}})$ on $(\mathrm{fs}/S)_{\mathrm{kfl}}$. The $\cO_{S_{\kfl}}$-algebra structure on $\cO_{S_{\kfl}}$ makes $\cA_{G}$ an $\cO_{S_{\kfl}}$-algebra, and the abelian group structure on $G$ defines a commutative Hopf $\cO_{S_{\mathrm{kfl}}}$-algebra structure on $\cA_{G}$. If $G\in (\fin/S)_{\rm{c}}$, $\cA_{G}$ coincides with the usual coordinate ring of $G$. Hence, by working Kummer log flat locally on $S$, we see that the underlying $\cO_{S_{\mathrm{kfl}}}$-module of $\cA_{G}$ belongs to $\mathcal{LLF}(S)$. Therefore, $\cA_{G}$ is a commutative Hopf object of the unitary symmetric monoidal category $\mathcal{LLF}(S)$.

Conversely, for a commutative Hopf object $\cA$ of $\mathcal{LLF}(S)$, we let $G_{\cA}$ denote the sheaf of homomorphisms of $\cO_{S_{\mathrm{kfl}}}$-algebras $\mathcal{H}om(\cA,\cO_{S_{\mathrm{kfl}}})$ on $(\mathrm{fs}/S)_{\mathrm{kfl}}$. The commutative Hopf structure on $\cA$ defines an abelian group structure $G_{\cA}$. By working  Kummer log flat locally on $S$, we see that $G_{\cA}$ is an object of $(\mathrm{fin}/S)_{\mathrm{f}}$. 

\begin{prop}[cf.~{\cite[Proposition 2.15]{kat23}}]\label{log coordinate ring} 

The assignments $G\mapsto \cA_{G}$ and $\cA\mapsto G_{\cA}$ give a contravariant equivalence between the category $(\fin/S)_{\rm{f}}$ and the category of commutative Hopf objects of the unitary symmetric monoidal category $\mathcal{LLF}(S)$. Moreover, an object $G\in (\fin/S)_{\rm{f}}$ is classical if and only if $\cA_{G}$ is classical.
\end{prop}

\begin{proof}
    The former assertion follows from the definition, and the latter assertion follows the former one. 
\end{proof}

\begin{prop}\label{sub of cl is cl}
    Let $H\subset G$ be objects of $(\fin/S)_{\rm{f}}$. Suppose that $G$ is classical. Then $H$ and $G/H$ are also classical.
\end{prop}

\begin{proof}
    By working Kummer log flat locally on $S$, we see that the inclusion morphism $H\hookrightarrow G$ induces a surjection $\cA_{G}\twoheadrightarrow \cA_{H}$. Since $\cA_{G}$ is classical, $\cA_{H}$ is also classical by \cite[Proposition 3.29]{niz08}. Hence, Proposition \ref{log coordinate ring} implies that $H$ is classical. Therefore, $G/H$ is also classical.
\end{proof}

\begin{construction}[Connected-\'{e}tale sequences, cf. {\cite[(2.6)]{kat23}}]\label{construction conn et seq} \noindent

Let $S$ be an fs log scheme such that $\mathring{S}=\mathrm{Spec}R$ for a henselian local ring $R$, and let $G\in (\fin/S)_{\rm{f}}$. We shall construct a unique exact sequence of objects of $(\fin/S)_{\rm{f}}$
\begin{equation}
    0\to G^{0}\to G\to G^{\et}\to 0
\end{equation}
such that, for every finite Kummer log flat cover $T\to S$ with $G|_{T}$ being classical, the pullback of the sequence (2.1) is the connected-\'{e}tale sequence for $G|_{T}$. This sequence is called the \emph{connected-\'{e}tale sequence} for $G$. The uniqueness follows from the uniqueness of connected-\'{e}tale sequences in non-log cases. 

First, we assume that $R$ is strict local. Take a chart $P\to \cM_{S}$. By Lemma \ref{log fin grp cl}, there exists an integer $n\geq 1$ such that the pullback $G|_{T}$ is classical, where we set $T\coloneqq S\times_{\bA_{P}} \bA_{P^{1/n}}$. Since the underlying schemes of $T$, $T\times_{S} T$, and $T\times_{S} T\times_{S} T$ are the spectrums of finite products of strict local rings, the connected-\'{e}tale sequence for $G|_{T}$ descends to the sequence (2.1) via Kummer log flat descent (Lemma \ref{kfl descent for weak log fin}) by the uniqueness of connected-\'{e}tale sequences in non-log cases.

In general, what we proved in the previous paragraph and the limit argument (Proposition \ref{lim log fin grp sch}) allow us to take a strict finite \'{e}tale cover $S'\to S$ such that $G|_{S'}$ admits a connected-\'{e}tale sequence. By the uniqueness of connected-\'{e}tale sequences, this descends to the desired sequence (2.1). 
\end{construction}

\begin{lem}\label{subquot and connet seq}

Let $S$ be an fs log scheme such that $\mathring{S}$ is the spectrum of a henselian local ring. Let $H\subset G$ be objects in $(\fin/S)_{\rm{f}}$. Then there are natural isomorphisms of objects in $(\fin/S)_{\rm{f}}$
\begin{align*}
    H^{0}\cong H\cap G^{0}, \  H^{\et}\cong \mathrm{Im}(H\to G^{\et}), \ (G/H)^{0}\cong G^{0}/H^{0}, \ (G/H)^{\et}\cong G^{\et}/H^{\et}.
\end{align*}
\end{lem}

\begin{proof}
    The analogous statements for non-log finite locally free group schemes are well-known, to which  one can reduce the original statement by Kummer log flat descent.
\end{proof}

\begin{prop}[cf.~{\cite[Proposition 2.7 (3)]{kat23}}]\label{conn et}
    Let $S$ be an fs log scheme such that $\mathring{S}=\Spec R$ for a henselian local ring $R$ with residue characteristic $p>0$. Let $G$ be an object of $(\fin/S)_{\rm{f}}$ killed by some power of $p$. Then the object $G$ belongs to $(\fin/S)_{\rm{d}}$ if and only if both of $G^{0}$ and $G^{\et}$ are classical. Moreover, in this case, $G^{0}$ (resp.~$G^{\et}$) is a connected finite locally free group scheme (resp.~a finite \'{e}tale group scheme) over $\mathring{S}$. 
\end{prop}

\begin{proof} This is proved in \cite[Proposition 2.7 (3)]{kat23} in the case where $R$ is noetherian. In general, one can reduce the problem to this case by the limit argument (Proposition \ref{lim log fin grp sch}).
\end{proof}

\begin{lem}\label{global conn et fin ver}
Let $G$ be a log finite group scheme over $S$. Suppose that $p$ is locally nilpotent on $\mathring{S}$ and that the order of $(G|_{s})^{\et}$ is constant for every $s\in S$. Then there exists a unique exact sequence
\[
0\to G^{0}\to G\to G^{\et}\to 0,
\]
where $G^{0}$ is a finite locally free radiciel group scheme over $\mathring{S}$ and $G^{\et}$ is a finite \'{e}tale group scheme over $\mathring{S}$.    
\end{lem}

\begin{proof}
Since the exact sequence as in the statement is unique if it exists, we may work strict \'{e}tale locally on $S$. Hence, by the limit argument (Proposition \ref{lim log fin grp sch}), we may assume that $\mathring{S}$ is the spectrum of a strict local ring. Consider the connected-\'{e}tale exact sequence 
\[
0\to G^{0}\to G\to G^{\et}\to 0,
\]
where $G^{0}$ is a finite locally free group scheme over $\mathring{S}$ and $G^{\et}$ is a finite \'{e}tale group scheme over $\mathring{S}$ by Proposition \ref{conn et}. For every point $t\in S$, we have the following diagram:
\[
\begin{tikzcd}
    0 \ar[r] & ( G|_{t})^{0} \ar[r] \ar[d,dashed] & G|_{t} \ar[r] \ar[d,equal] &  (G|_{t})^{\et} \ar[r] \ar[d,dashed] & 0 \\
    0 \ar[r] & (G^{0})|_{t} \ar[r] & G|_{t} \ar[r] & (G^{\et})|_{t} \ar[r] & 0.
\end{tikzcd}
\]
By the assumption, the right vertical homomorphism is a surjection between finite \'{e}tale group schemes of the same order, which is an isomorphism. Hence, the left vertical map is also an isomorphism. This proves that $G^{0}$ is radiciel over $\mathring{S}$. 
\end{proof}

\begin{prop}[cf.~{\cite[Proposition 2.16]{kat23}}]\label{fiberwise}
    Let $G\in (\fin/S)_{\rm{f}}$. Let $\star\in \{\mathrm{d,c}\}$ be a subscript. Suppose that the pullback of $G$ to a point $s\in S$ belongs to $(\fin/s)_{\star}$. Then there exists an open neighborhood $U\subset S$ of $s$ such that $G|_{U}$ belongs to $(\fin/U)_{\star}$. In particular, if the pullback of $G$ to every point $s\in S$ belongs to $(\fin/s)_{\star}$, then $G$ also belongs to $(\fin/S)_{\star}$.
\end{prop}

\begin{proof}
    The limit argument (Proposition \ref{lim log fin grp sch}) allows us to assume that $S$ is locally noetherian, in which case the assertion is proved in \cite[Proposition 2.16]{kat23}. 
\end{proof}

\begin{lem}\label{prime order log fin grp}
Let $G$ be a log finite group scheme over $S$ of order $p$. Suppose that $p$ is locally nilpotent over $\mathring{S}$. Then $G$ is classical.
\end{lem}

\begin{proof}
    Due to Proposition \ref{fiberwise}, we may assume that the $\mathring{S}$ is the spectrum of a field. Since the order of $G$ is equal to the product of the orders of $G^{0}$ and $G^{\et}$, the log finite group scheme $G$ is isomorphic to either $G^{0}$ or $G^{\et}$. Hence, the assertion follows from \cite[Proposition 2.7 (3)]{kat23}.
\end{proof}

\begin{lem}\label{subgrp uniqueness}
    Let $S$ be an fs log scheme such that $\mathring{S}=\mathrm{Spec} R$ for a reduced henselian local ring $R$ with residue characteristic $p>0$. Let $G$ be a log finite group scheme over $S$ killed by some power of $p$, and let $H_{1}, H_{2}$ be log finite subgroup schemes of $G$. Suppose that $H_{1}^{0}$ is contained in $H_{2}^{0}$, $H_{1}^{\et}$ is contained in $H_{2}^{\et}$, and $G^{0}$ is radiciel over $S$. Then $H_{1}$ is contained in $H_{2}$.
\end{lem}

\begin{proof}
By Lemma \ref{subquot and connet seq}, we have the following commutative diagram of log finite group schemes:
    \[
     \begin{tikzcd}
     0 \ar[r] & H_{1}^{0} \ar[r] \ar[d,"0"] & H_{1} \ar[r] \ar[d] & H_{1}^{\et} \ar[r] \ar[d,"0"] & 0
     \\
     0 \ar[r] & G^{0}/H_{2}^{0} \ar[r] & G/H_{2} \ar[r] & G^{\et}/H_{2}^{\et} \ar[r] & 0.
     \end{tikzcd}
     \]
     By the above diagram, the canonical morphism $H_{1}\to G/H_{2}$ factors as $H_{1}\to H_{1}^{\et}\to G^{0}/H_{2}^{0}\to G/H_{2}$. Since $H_{1}^{\et}$ is reduced and $G^{0}/H_{2}^{0}$ is radiciel over $\mathring{S}$, we get
     \[
     \Hom_{S}(H_{1}^{\et}, G^{0}/H_{2}^{0})=0. 
     \]
     Therefore, $H_{1}$ is contained in $H_{2}$.
\end{proof}

\begin{cor}\label{sub of d is d}
    Let $H\subset G$ be objects of $(\fin/S)_{\rm{f}}$. Suppose that $G$ belongs to $(\fin/S)_{\rm{d}}$. Then $H$ and $G/H$ also belong to $(\fin/S)_{\rm{d}}$.
\end{cor}

\begin{proof}
    Proposition \ref{fiberwise} allows us to assume that $\mathring{S}=\mathrm{Spec} k$ for a field $k$ of characteristic $p\geq 0$. By \cite[Proposition 2.1]{kat23}, it is enough to treat the case where $p>0$ and $G$ is killed by some power of $p$. 

    Note that $G^{0}$ and $G^{\et}$ are classical by Proposition \ref{conn et}. There are inclusion relations $H^{0}\subset G^{0}$ and $H^{\et}\subset G^{\et}$ by Lemma \ref{subquot and connet seq}. Hence, Proposition \ref{sub of cl is cl} implies that $H^{0}$ and $H^{\et}$ are also classical. By using Proposition \ref{conn et} again, we see that $H$ belongs to $(\fin/S)_{\rm{d}}$. By applying this result to the subobject $(G/H)^{*}\subset G^{*}$, we conclude that $G/H$ also belongs to $(\fin/S)_{\rm{d}}$.
\end{proof}

\subsection{Log \texorpdfstring{$p$}--divisible groups}

\begin{dfn}[\cite{kat23}, Section 4]
    Let $(\BT/S)_{\rm{f}}$ be the category of sheaves of abelian groups $G$ on $(\fs/S)_{\kfl}$ satisfying the following conditions.
    \begin{enumerate}
        \item The equation $\displaystyle G=\bigcup_{n\geq 1} G[p^n]$ holds, where $G[p^n]:=\Ker(\times p^n\colon G\to G)$.
        \item The morphism $\times p\colon G\to G$ is surjective.
        \item For each $n\geq 1$, the sheaf $G[p^n]$ belongs to $(\fin/S)_{\rm{f}} $.
    \end{enumerate}
    Furthermore, for a subscript $\star\in \{\mathrm{d,c}\}$, we define a category $(\BT/S)_{\star}$ as the full subcategory of $(\BT/S)_{\rm{f}}$ consisting of objects $G$ such that $G[p^{n}]$ belongs to $(\fin/S)_{\star}$ for each $n\geq 1$. We say that $G\in (\mathrm{BT}/S)_{\mathrm{f}}$ is \emph{classical} is $G$ belongs to $(\mathrm{BT}/S)_{\mathrm{c}}$. The category $(\BT/S)_{\rm{c}}$ is naturally equivalent to the category of $p$-divisible groups over $\mathring{S}$. Objects in $(\BT/S)_{\rm{d}}$ are called \emph{log $p$-divisible groups} over $S$.
\end{dfn}

\begin{dfn}
    For $G\in (\BT/S)_{\rm{f}}$, we say that $G$ is \emph{of height $h$} if $G[p^n]$ is of order $p^{nh}$ for each $n\geq 1$ (equivalently, $G[p]$ is of order $p^h$).
\end{dfn}

\begin{ex}
    Let $n\geq 1$. For any log abelian variety $A$ over $S$ with $\mathrm{dim}(A)=d$, the object $\displaystyle A[p^{\infty}]\coloneqq \bigcup_{n\geq 1} A[p^n]$ is a log $p$-divisible group of height $2d$ over $S$ because $A[p^n]$ is a log finite group scheme of order $p^{2nd}$ by Example \ref{log av torsion} and $\times p\colon A\to A$ is surjective by \cite[Subsection 18.6]{kkn15}.
\end{ex}

\begin{lem}
    Let $S$ be an fs log scheme such that $\mathring{S}$ is the spectrum of a henselian local ring. Let $G\in (\BT/S)_{\rm{f}}$. Then
    \[
    G^{0}\coloneqq \bigcup_{n\geq 1} G[p^{n}]^{0} \ \  \text{and} \ \  G^{\et}\coloneqq \bigcup_{n\geq 1} G[p^{n}]^{\et}
    \]
    are objects of $(\BT/S)_{\rm{f}}$, and there exists an exact sequence
    \[
    0\to G^{0}\to G\to G^{\et}\to 0.
    \]
\end{lem}

\begin{proof}
    Let $\star\in \{0,\et\}$ be a subscript and $n,m\geq 1$ be integers. By Lemma \ref{subquot and connet seq}, the homomorphism $\times p^{n}\colon G[p^{n+m}]^{\star}\to G[p^{n+m}]^{\star}$ factors as
    \[
    G[p^{n+m}]^{\star}\to G[p^{m}]^{\star}\hookrightarrow G[p^{n+m}]^{\star}.
    \]
    We also simply write $\times p^{n}$ for the homomorphism $G[p^{n+m}]^{\star}\to G[p^{m}]^{\star}$ defined above. Then, by applying Lemma \ref{subquot and connet seq} again, we see that the sequence
    \[
    0\to G[p^{n}]^{\star}\hookrightarrow G[p^{n+m}]^{\star}\stackrel{\times p^{n}}{\to} G[p^{m}]^{\star}\to 0
    \]
    is exact, which implies that $G^{\star}$ is an object of $(\BT/S)_{\rm{f}}$. The remaining assertion is clear.
\end{proof}

\begin{lem}\label{global conn et bt ver}
Let $G$ be a log $p$-divisible group over $S$. Suppose that $p$ is locally nilpotent on $S$ and that the height of $(G|_{s})^{\et}$ is constant for every $s\in S$. Then there exists a unique exact sequence
\[
0\to G^{0}\to G\to G^{\et}\to 0,
\]
where $G^{0}$ is a formal Lie group over $\mathring{S}$ and $G^{\et}$ is an \'{e}tale $p$-divisible group over $\mathring{S}$.    
\end{lem}

\begin{proof}
    This immediately follows from Lemma \ref{global conn et fin ver}.
\end{proof}

\begin{dfn}
     Let $G_{1}$ and $G_{2}$ be log $p$-divisible groups over $S$. A homomorphism $f\colon G_{1}\to G_{2} $ is called an \emph{isogeny} if $f$ is surjective and $\Ker(f)$ belongs to $(\mathrm{fin}/S)_{\mathrm{f}}$.
     
     Let $\mathrm{QIsog}(G_{1},G_{2})$ be the subset of $\mathrm{Hom}(G_{1},G_{2})[1/p]$ consisting of elements which can be written as $g/p^{n}$ for an integer $n\geq 1$ and an isogeny $g\colon G_{1}\to G_{2}$. A \emph{quasi-isogeny} from $G_{1}$ to $G_{2}$ is an element of the global section of the sheaf on the Zariski site of $S$ associated with the presheaf given by $U\mapsto \mathrm{QIsog}(G_{1}|_{U},G_{2}|_{U})$ for an open subset $U\subset S$. Note that, when $S$ is quasi-compact, the set of quasi-isogenies from $G_{1}$ to $G_{2}$ is naturally in bijection with the set $\mathrm{QIsog}(G_{1},G_{2})$.
\end{dfn}

\begin{lem}
    Let $G$ be a log $p$-divisible group and $H$ be a log finite subgroup scheme of $G$. Then the sheaf $G/H$ is a log $p$-divisible group, and the canonical map $G\to G/H$ is an isogeny. Moreover, this gives a one-to-one correspondence between log finite subgroup schemes of $G$ and isogenies from $G$.
\end{lem}

\begin{proof}
    By working Zariski locally on $S$, we may assume that $H$ is killed by $p^{d}$ for an integer $d\geq 1$. For an integer $n\geq d$, applying the snake lemma to the diagram
    \[
    \begin{tikzcd}
        0 \ar[r] & H \ar[r] \ar[d,"0"] & G \ar[r] \ar[d,"\times p^{n}"] & G/H \ar[r] \ar[d,"\times p^{n}"] & 0 \\
        0 \ar[r] & H \ar[r] & G \ar[r] & G/H \ar[r] & 0 
    \end{tikzcd}
    \]
    gives an exact sequence
    \[
    0\to G[p^{n}]/H\to (G/H)[p^{n}]\to H\to 0.
    \]
    Hence, it follows from \cite[Proposition 2.3]{kat23} that $(G/H)[p^{n}]$ is a log finite group scheme, and so $G/H$ is a log $p$-divisible group.
    
    Conversely, for an isogeny $f\colon G_{1}\to G_{2}$, the object $\mathrm{Ker}(f)\in (\mathrm{fin}/S)_{\mathrm{f}}$ is a log finite group scheme by Corollary \ref{sub of d is d}. This proves the assertion.
\end{proof}

\begin{prop}\label{weaker def of isog}
Let $f\colon G_{1}\to G_{2}$ be a surjection of log $p$-divisible groups over a locally noetherian fs log scheme $S$. Suppose that, Zariski locally, $\mathrm{Ker}(f)$ is killed by $p^{d}$ for an integer $d\geq 1$. Then $f$ is an isogeny.
\end{prop}

\begin{proof}
    It is enough to prove that $\mathrm{Ker}(f)$ is a log finite group scheme over $S$. By working Zariski locally on $S$, we may assume that $\mathrm{Ker}(f)\subset G_{1}[p^{d}]$ for an integer $d\geq 1$. The sheaf $\mathrm{Ker}(f)$ is representable by an fs log scheme of finite presentation over $S$ as it coincides with the kernel of the homomorphism of log finite group schemes $G_{1}[p^{d}]\to G_{2}[p^{d}]$ induced from $f$. Hence, by the limit argument (Proposition \ref{lim log fin grp sch} and Proposition \ref{lim log sch}), we may assume that $\mathring{S}$ is the spectrum of a strict local ring.

    Since $\mathrm{Ker}(f)$ is killed by $p^{d}$ and $f$ is surjective, there exists a unique homomorphism $g\colon G_{2}\to G_{1}$ such that both of $fg$ and $gf$ are equal to the multiplication by $p^{d}$. The homomorphism $f$ uniquely induces homomorphisms $f^{0}\colon G_{1}^{0}\to G_{2}^{0}$ and $f^{\et}\colon G_{1}^{\et}\to G_{2}^{\et}$ fitting into the following diagram:
    \[
    \begin{tikzcd}
        0 \ar[r] & G_{1}^{0} \ar[r] \ar[d,"f^{0}"] & G_{1} \ar[r] \ar[d,"f"] & G_{1}^{\et} \ar[r] \ar[d,"f^{\et}"] & 0 \\
        0 \ar[r] & G_{2}^{0} \ar[r] & G_{2} \ar[r] & G_{2}^{\et} \ar[r] & 0.
    \end{tikzcd}
    \]
    Similarly, the homomorphism $g$ uniquely induces homomorphisms $g^{0}\colon G_{2}^{0}\to G_{1}^{0}$ and $g^{\et}\colon G_{2}^{\et}\to G_{1}^{\et}$. By the uniqueness of the induced morphism on the connected part and the \'{e}tale part, all of the homomorphisms $f^{0}g^{0}$, $f^{\et}g^{\et}$, $g^{0}f^{0}$, and $g^{\et}f^{\et}$ are the multiplication by $p^{d}$. It follows from \cite[Proposition 3.3.8]{cco14} that all of the homomorphisms $f^{0}$, $f^{\et}$, $g^{0}$, and $g^{\et}$ are isogenies. By applying the snake lemma to the above diagram, we get an exact sequence
    \[
    0\to \mathrm{Ker}(f^{0})\to \mathrm{Ker}(f)\to \mathrm{Ker}(f^{\et})\to 0.
    \]
    Since $\mathrm{Ker}(f^{0})$ and $\mathrm{Ker}(f^{\et})$ are finite locally free group schemes over $\mathring{S}$, the sheaf $\mathrm{Ker}(f)$ is a log finite group scheme over $S$ by \cite[Proposition 2.3]{kat23}.
\end{proof}

\section{Completely slope divisible log \texorpdfstring{$p$}--divisible groups}\label{sec 3}

Throughout this section, let $S$ be an fs log scheme over $\bF_{p}$.

\begin{dfn}
    Let $G$ be a sheaf of abelian groups on $(\mathrm{fs}/S)_{\mathrm{kfl}}$ (the examples of primary interest are objects of $(\mathrm{fin}/S)_{\mathrm{f}}$ or $(\mathrm{BT}/S)_{\mathrm{f}}$). Let $G^{(p^{n})}$ denote the pullback of $G$ by the $n$-th power Frobenius morphism on $S$.
    
    For an fs log scheme $T$ over $S$, we write $T^{(F)}$ for the fs log scheme $T$ equipped with the structure map $T\to S\stackrel{F}{\to} S$, where $F$ is the Frobenius morphism on $S$. The Frobenius morphism on $T$ is a map $T^{(F)}\to T$ over $S$, and so we have an induced homomorphism $G(T)\to G(T^{(F)})\cong G^{(p)}(T)$. This defines a homomorphism $F\colon G\to G^{(p)}$, called the \emph{Frobenius homomorphism}. When $G$ belongs to $(\mathrm{fin}/S)_{\mathrm{c}}$ or $(\mathrm{BT}/S)_{\mathrm{c}}$, the homomorphism $F$ is nothing but the usual Frobenius homomorphism. For an integer $n\geq 1$, the composite of homomorphisms
    \[
    G\stackrel{F}{\to} G^{(p)}\stackrel{F^{(p)}}{\to} G^{(p^{2})}\stackrel{F^{(p^{2})}}{\to} \dots \stackrel{F^{(p^{n-1})}}{\to} G^{(p^{n})}
    \]
    are simply denoted by $F^{n}$, where $F^{(p^{k})}$ denotes the base change of $F$ by the $k$-th power Frobenius morphism on $S$ for an integer $k\geq 1$.
\end{dfn}

\begin{lem}\label{frob nilp imlply classical}
    Let $G$ be a log finite group scheme over $S$. Suppose that there is an integer $n\geq 1$ such that the homomorphism $F^{n}\colon G\to G^{(p^{n})}$ is an isomorphism (resp.~a zero map). Then $G$ is classical. In particular, $G$ is a finite \'{e}tale group scheme (resp.~a finite radiciel locally free group scheme) over $\mathring{S}$. 
\end{lem}

\begin{proof}
    We shall prove the assertion just in the case where $F^{n}$ is an isomorphism because the same argument also works in the other case. By Proposition \ref{fiberwise}, we may assume that $\mathring{S}=\mathrm{Spec} k$ for a field $k$. In this case, we have the following diagram:
    \[
    \begin{tikzcd}
        0 \ar[r] & G^{0} \ar[r] \ar[d,"F^{n}"] & G \ar[r] \ar[d,"F^{n}","\sim"' sloped] & G^{\et} \ar[r] \ar[d,"F^{n}","\sim"' sloped] & 0 \\
        0 \ar[r] & (G^{0})^{(p^{n})} \ar[r] & G^{(p^{n})} \ar[r] & (G^{\et})^{(p^{n})} \ar[r]  & 0. 
    \end{tikzcd}
    \]
    Since $G^{\et}$ is a finite \'{e}tale group scheme over $\mathring{S}$, the right vertical map in this diagram is an isomorphism. Hence, the left vertical map is also an isomorphism. On the other hand,  the $m$-th power of the Frobenius homomorphism $G^{0}\to (G^{0})^{(p^{m})}$ is a zero map for some integer $m\geq 1$ because $G^{0}$ is a connected finite group scheme over the field $k$. Therefore, $G^{0}$ is trivial, and so $G\cong G^{\et}$. In particular, $G$ is classical. 
\end{proof}

\begin{dfn} \label{log csd}
  A log $p$-divisible group $G$ over $S$ is \emph{completely slope divisible} if $G$ admits a filtration of log $p$-divisible subgroups
  \[
  0=G_0\subset G_1 \subset \dots \subset G_m=G,
  \]
  and if there are integers
  \[
  s\geq r_1> r_2 >\dots >r_{m}\geq 0
  \]
  such that the following conditions are satisfied :
  \begin{enumerate}
      \item The quasi-isogeny $p^{-r_{i}}F^{s}$ from $G_{i}$ to $G_{i}^{(p^{s})}$ is an isogeny for every $1\leq i\leq m$.
      \item The morphism $G_{i}/G_{i-1}\to (G_{i}/G_{i-1})^{(p^{s})}$ induced from $p^{-r_{i}}F^{s}\colon G_{i}\to G_{i}^{(p^{s})}$ is an isomorphism for every $1\leq i\leq m$.
  \end{enumerate}
\end{dfn}

\begin{prop}\label{dec of csd}
    Let $G$ be a completely slope divisible log $p$-divisible group over $S$. Then there exists the unique exact sequence
    \[
    0\to G^{0}\to G\to G^{\et}\to 0,
    \]
    where $G^{0}$ is a completely slope divisible formal Lie group over $\mathring{S}$ and $G^{\et}$ is an \'{e}tale $p$-divisible group over $\mathring{S}$.
\end{prop}

\begin{proof}
We use the notation in Definition \ref{log csd}. We may assume that $r_{m}=0$. Let $n\geq 1$ be an integer. By the definition, the Frobenius homomorphism $F\colon (G/G_{m-1})[p^{n}]\to (G/G_{m-1})[p^{n}]^{(p)}$ is an isomorphism, and the homomorphism $F^{sn}\colon G_{m-1}[p^{n}]\to G_{m-1}[p^{n}]^{(p^{sn})}$ is a zero map. Hence, by Lemma \ref{frob nilp imlply classical}, $G/G_{m-1}$ is an \'{e}tale $p$-divisible group over $\mathring{S}$, and $G_{m-1}$ is a formal Lie group over $\mathring{S}$. Therefore, the exact sequence
\[
0\to G_{m-1}\to G\to G/G_{m-1}\to 0
\]
is the desired one.
\end{proof}

We generalize the notion of Newton polygons to log $p$-divisible groups.

\begin{dfn}
    Let $S$ be an fs log scheme such that $\mathring{S}$ is the spectrum of a field of characteristic $p>0$. Let $G$ be a log $p$-divisible group over $S$. The \emph{Newton polygon} of $G$ is the Newton polygon of the classical $p$-divisible group $G^{0}\oplus G^{\et}$.
\end{dfn}

The notion of Newton polygons for log $p$-divisible groups enables us to formulate a log version of \cite[Proposition 2.3]{oz02} as follows.

\begin{cor}\label{generically csd implies csd}
    Let $G$ be a log $p$-divisible group over $S$ with constant Newton polygon. Suppose that $\mathring{S}$ is integral, and that the pullback $G|_{\eta}$ is completely slope divisible, where $\eta$ is the generic point. Then $G$ is completely slope divisible. 
\end{cor}

\begin{proof}
    We take $G^{0}$ and $G^{\et}$ as in Lemma \ref{global conn et bt ver}. The formal Lie group $G^{0}$ has a constant Newton polygon, and its generic fiber $G^{0}|_{\eta}$ is completely slope divisible. Hence, \cite[Proposition 2.3]{oz02} implies that $G^{0}$ is completely slope divisible. Therefore, $G$ is also completely slope divisible.
\end{proof}

The following corollary is a log version of \cite[Theorem 2.1]{oz02}.

\begin{cor}\label{first main thm}
    Let $G$ be a log $p$-divisible group over $S$ with constant Newton polygon. Suppose that $\mathring{S}$ is locally noetherian normal. Then there exist a completely slope divisible log $p$-divisible group $H$ over $S$ and an isogeny $G\to H$. 
\end{cor}

\begin{proof}
    We take $G^{0}$ and $G^{\et}$ as in Lemma \ref{global conn et bt ver}. Since $G^{0}$ has a constant Newton polygon, there exist a completely slope divisible classical $p$-divisible group $H'$ and an isogeny $\phi\colon G^{0}\to H'$ by \cite[Theorem 2.1]{oz02}. Then there is an exact sequence
    \[
    0\to H'\to G/\Ker(\phi)\to G^{\et}\to 0,
    \]
    and so $G/\Ker(\phi)$ is a completely slope divisible log $p$-divisible group. Hence, the canonical surjection $G\twoheadrightarrow G/\Ker(\phi)$ is the desired isogeny.
\end{proof}

The following lemma is a log version of \cite[Lemma 2.5]{oz02}, which we use in the proof of the main theorem.

\begin{lem}\label{fin csd quot}
    Suppose that $\mathring{S}$ is the spectrum of an algebraically closed field $k$. Let $G$ be a log $p$-divisible group over $S$ and $d\geq 1$ be an integer. Then there are only finitely many isomorphism classes of isogenies $G\to H$ of degree $p^d$ in which $H$ is a completely slope divisible log $p$-divisible group.
\end{lem}

\begin{proof}
    Let $K$ be a log finite subgroup scheme of $G$ of order $p^{d}$. Since $G/K$ is completely slope divisible if and only if $(G/K)^{0}=G^{0}/K^{0} $ and $(G/K)^{\text{\'{e}t}}=G^{\text{\'{e}t}}/K^{\text{\'{e}t}} $, there are only finitely many possibilities of $K^{0}$ and $K^{\text{\'{e}t}}$ such that $G/K$ is completely slope divisible by \cite[Lemma 2.5]{oz02}. Lemma \ref{subgrp uniqueness} implies that $K$ is determined from $K^{0}$ and $K^{\text{\'{e}t}}$ uniquely. This proves the statement.
\end{proof}

\section{Some extension properties on log regular schemes}\label{sec 4}

In this section, we study log finite groups and log $p$-divisible groups over log regular bases.

\subsection{Log regular schemes}

In this subsection, we recall the definition of log regularity and some properties of log regular log schemes.

\begin{dfn}[\cites{kat94,niz06}]
    Let $S$ be a locally noetherian fs log scheme. Let $s\in S$ and let $\bar{s}$ be a geometric point on $s$. Let $I(\bar{s})$ be the ideal of $\cO_{S, \bar{s}}$ generated by the image of the map $\cM_{S, \bar{s}}\setminus \cO_{S, \bar{s}}^{\times}\to \cO_{S, \bar{s}}$. We say that $S$ is \emph{log regular} at $s$ if the following two conditions are satisfied:
    \begin{enumerate}
        \item $\cO_{S, \bar{s}}/I(\bar{s})$ is a regular local ring;
        \item $\dim(\cO_{S, \bar{s}})=\dim(\cO_{S, \bar{s}}/I(\bar{s}))+\rk(\cM_{S, \bar{s}}^{\mathrm{gp}}/\cO_{S, \bar{s}}^{\times})$.
    \end{enumerate}
    The log scheme $S$ is called \emph{log regular} if it is log regular at every point $s\in S$. 
\end{dfn}

\begin{prop}[Kato]\label{log reg}
    For a locally noetherian fs log scheme $S$, the following statements are true. 
    
    (1) The subset of $S$ defined as \{$s\in S$ $|$ $S$ is log regular at $s$\} is stable under generalization.
    
    (2) If $S$ is log regular at $s\in S$, $\mathring{S}$ is normal at $s$.
\end{prop}

\begin{proof}

    (1) See \cite[Proposition 7.1]{kat94}.
    
    (2) See \cite[Theorem 4.1]{kat94}.
\end{proof}

\begin{lem} \label{log reg stab}
    Let $S$ be a log regular fs log scheme such that $\mathring{S}=\Spec R$ for a noetherian strict local ring $R$. Let $s$ be the unique closed point of $S$. Fix a chart $P\to \cM_S$ which is neat at $s$. Then, for a sharp fs monoid $Q$ with a Kummer map $P\to Q$, the log scheme $T\coloneqq S\times_{\bA_{P}} \bA_{Q}$ is also log regular.
\end{lem}

\begin{proof}
    First, we claim that $R\otimes_{\bZ[P]} \bZ[Q]$ is a noetherian strict local ring. Since $R\otimes_{\bZ[P]} \bZ[Q]$ is finite over $R$, it suffices to prove that the special fiber of $T\to S$ consists of a single point. Let $k$ be the residue field of $R$. Since $Q$ is sharp and $P\to Q$ is injective, sending $Q\backslash \{1\}$ to $0$ defines a ring map $k\otimes_{\bZ[P]} \bZ[Q] \to k$. This ring map induces a ring map $(k\otimes_{\bZ[P]} \bZ[Q])_{\mathrm{red}} \to k$, which is the inverse of the natural map $k\to (k\otimes_{\bZ[P]} \bZ[Q])_{\mathrm{red}}$. Hence, we get an isomorphism $k\cong (k\otimes_{\bZ[P]} \bZ[Q])_{\mathrm{red}}$, and so the special fiber of $T\to S$ consists of a single point.
    
    By Proposition \ref{log reg} (1), it suffices to prove that $T$ is log regular at the unique closed point $t$. Since $R\to R\otimes_{\bZ[P]} \bZ[Q]$ is finite and injective, we have the equation $\mathrm{dim}(R)=\mathrm{dim}(R\otimes_{\bZ[P]} \bZ[Q])$. Since $P\to Q$ is Kummer, the following equations hold: 
    \[
    \mathrm{rk}(\cM_{S,s}^{\mathrm{gp}}/\cO_{S,s}^{\times})=\mathrm{rk}(P^{\mathrm{gp}})=\mathrm{rk}(Q^{\mathrm{gp}})=\mathrm{rk}(\cM_{T,t}^{\mathrm{gp}}/\cO_{T,t}^{\times}).
    \]
    Since $S$ is log regular, it is enough to show the equation
    \begin{equation}
        \mathrm{dim}(R/I(s))=\mathrm{dim}((R\otimes_{\bZ[P]} \bZ[Q])/I(t)).
    \end{equation}
    The ideal $I(s)\subset R$ (resp. $I(t)\subset R\otimes_{\bZ[P]} \bZ[Q]$) is generated by the image of $P\backslash \{1\}$ (resp. $Q\backslash \{1\}$). Sending $Q\backslash \{1\}$ to $0$ defines a ring morphism $(R\otimes_{\bZ[P]} \bZ[Q])/I(t)\to R/I(s)$, and it can be checked directly that this map is the inverse of the natural map $R/I(s)\to (R\otimes_{\bZ[P]} \bZ[Q])/I(t)$. In particular, we obtain the equation (4.1).
\end{proof}

\subsection{one-dimensional cases}

\begin{lem} \label{log reg dvr}
    Let $S$ be a log regular fs log scheme such that $\mathring{S}=\Spec R$ for a strict local discrete valuation ring $R$. Then the log structure of $S$ is either of the trivial one or the standard one (i.e. the log structure defined by the unique closed point).
\end{lem}

\begin{proof}
     Let $s$ denote the unique closed point of $S$. The condition $\rk((\overline{\cM_{S, s}})^{\mathrm{gp}})\leq 1$ implies that the monoid $\overline{\cM_{S, s}}$ is isomorphic to either of $0$ or $\bN$. In the former case, the log structure of $S$ is trivial. In the latter case, we can take a chart $\bN\to \cM_{S, s}$ which is neat at $s$. Since $R/I(s)$ is a field, the morphism $\bN\to \cM_{S, s}\to R$ maps $1$ to a uniformizer of $R$. Hence, the log structure of $S$ is the standard one.
\end{proof}

\begin{prop}\label{ext of subgrp over log reg dvr}
    Let $S$ be a log regular fs log scheme such that $\mathring{S}=\Spec R$ for a henselian discrete valuation ring $R$. Let $K$ be the fraction field of $R$. Let $G$ be a log finite group scheme over $S$ and $H_{K}$ be a log finite subgroup scheme of $G|_{K}$. Then $H_{K}$ uniquely extends to a log finite subgroup scheme of $G$.
\end{prop}

\begin{proof}
    The limit argument (Proposition \ref{lim log fin grp sch}) and Galois descent allow us to assume that $R$ is strictly local. By Lemma \ref{log reg dvr}, the log structure of $S$ is either of the trivial one or the standard one. In the former case, $G$ is classical, and so the scheme theoretic closure of $H_{K}$ in $G$ is the unique extension of $H_{K}$.

    It is enough to consider the latter case. We fix a uniformizer $\pi\in R$. By Lemma \ref{log fin grp cl}, there exists an integer $n\geq 1$ such that $G|_{T}$ is classical, where $T$ is $\Spec R[\pi^{1/n}]$ equipped with the standard log structure. Let $L\coloneqq K(\pi^{1/n})$ be the fraction field of $R[\pi^{1/n}]$. Let $T'\coloneqq T\times_{S} T$. We have the following commutative diagram in which every square is a cartesian diagram in the category of saturated log schemes:
    \[
    \begin{tikzcd}
        T' \ar[d,xshift=0.5ex] \ar[d,xshift=-0.5ex] & \mathrm{Spec}(L\otimes_{K} L) \ar[l,hook'] \ar[d,xshift=0.5ex] \ar[d,xshift=-0.5ex]  \\
        T \ar[d] & \mathrm{Spec}L \ar[l,hook'] \ar[d] \\
        S & \mathrm{Spec}K \ar[l,hook'].
    \end{tikzcd}
    \]
    Here, all log schemes in the right column are endowed with the trivial log structure, and all horizontal arrows are strict open immersion.
    
    By taking the scheme theoretic closure, we see that $H_{K}|_{L}\subset G|_{L}$ uniquely extends to a classical finite subgroup scheme $H_{T}$ of $G|_{T}$. As the pullbacks of $H_{T}$ by two projections $T'\to T$ coincide on the dense open subscheme $\Spec (L\otimes_{K} L)$ as subobjects of $G|_{L\otimes_{K} L}$, they coincide on the whole of $T'$. Hence, by Kummer log flat descent, $H_{T}$ descends to the subobject $H\subset G$ in the category $(\mathrm{fin}/S)_{\mathrm{f}}$ with an isomorphism $H_{K}\cong H|_{K}$. By Corollary \ref{sub of d is d}, $H$ is a log finite group scheme.
\end{proof}

\begin{rem}\label{need log reg}
Proposition \ref{ext of subgrp over log reg dvr} does not hold without the assumption that $S$ is log regular. Let $R$ be a henselian discrete valuation ring over $\bF_{p}$ with a uniformizer $\pi$, and $K$ be the fraction field of $R$. Let $S$ be the fs log scheme $\mathrm{Spec}R$ equipped with the log structure defined by $\bN\to R$ mapping $1$ to $\pi^{p}$ (which is not log regular). Let $\eta$ be the generic point of $S$ and $s$ be the closed point of $S$. By abuse of notation, we simply write $s$ for the log scheme $(\mathrm{Spec}k(s),\cM_{\mathrm{Spec}k(s)})$. As in Example \ref{log fin grp ass to q}, we can associate to $q\coloneqq \pi^{p}\in \bG_{m, \log}(S)$ an exact sequence of log finite group schemes over $S$
\begin{equation}
    0\to \mu_{p}\to G_{q}\to \bZ/p\bZ\to 0.
\end{equation}
The Kummer sequence
\[
0\to \mu_{p}\to \bG_{m,\mathrm{log}}\to \bG_{m,\mathrm{log}}\to 0
\]
induces the following commutative diagram in which both rows are exact:
    \[
     \begin{tikzcd}
      \pi^{p\bZ}R^{\times} \ar[r,"(-)^{p}"] \ar[d,hook] & \pi^{p\bZ}R^{\times} \ar[r] \ar[d,hook] \ar[rd,"0"] & \Ext^{1}_{(\mathrm{fs}/S)_{\mathrm{kfl}}}(\bZ/p\bZ, \mu_{p}) \ar[d] 
     \\
     K^{\times} \ar[r,"(-)^{p}"] & K^{\times} \ar[r] & \Ext^{1}_{(\mathrm{fs}/\eta)_{\mathrm{kfl}}}(\bZ/p\bZ, \mu_{p}).
     \end{tikzcd}
     \]
     By this diagram, we see that the generic fiber of the exact sequence (4.2) splits. We claim that the unique section $(\bZ/p\bZ)_{\eta}\rightarrowtail (G_{q})|_{\eta}$ does not extend to a log finite subgroup scheme of $G_{q}$. Assume that there exists such an extension $H\subset G_{q}$. Since the order of $H$ is a prime number $p$, $H$ is classical by Lemma \ref{prime order log fin grp}. Since the surjection $G_{q}\twoheadrightarrow \bZ/p\bZ$ admits no section, the composite $H\hookrightarrow G_{q}\twoheadrightarrow \bZ/p\bZ$ is not an isomorphism. Hence, the composite $H|_{s}\rightarrowtail G_{q}|_{s}\twoheadrightarrow (\bZ/p\bZ)_{s}$ is not an isomorphism, and so it is trivial. Therefore, we obtain $H|_{s}\cong \mu_{p,s}$ by (4.2). In particular, the multiplicative rank of $H$ jumps under the specialization. This is contradiction.
\end{rem}


\subsection{Extension properties on log regular schemes}

\begin{lem}\label{bit general purity}
    Let $X$ be a locally noetherian normal scheme and $f\colon Y\to X$ be a flat morphism of schemes. Let $U$ be a dense open subset of $X$ containing all points on $X$ of codimension $1$ and $V\coloneqq f^{-1}(U)$. Then the restriction functors
    \begin{align*}
        \mathcal{LF}(Y)&\to \mathcal{LF}(V) \\
        (\mathrm{fin}/Y)&\to (\mathrm{fin}/V)
    \end{align*}
    are fully faithful.
\end{lem}

\begin{proof}
    It is enough to prove the fully faithfulness of the former functor. By taking internal homomorphisms, the problem is reduced to showing that, for a vector bundle $\cE$ on $Y$, the restriction map $\Gamma(Y,\cE)\to \Gamma(V,\cE)$
    is an isomorphism. Let $i\colon V\hookrightarrow Y$ and $j\colon U\hookrightarrow X$ be natural open immersions. Then we have isomorphisms of $\cO_{Y}$-modules
    \begin{align*}
        i_{*}i^{*}\cE\cong \cE\otimes i_{*}\cO_{V}\cong \cE\otimes f^{*}j_{*}\cO_{U}\cong \cE\otimes f^{*}\cO_{X}\cong \cE,
    \end{align*}
    where the first isomorphism is the projection formula, the second one is the flat base change, and the third one follows from the assumption that $U$ is an open subset of a locally noetherian normal scheme $X$ containing all points of codimension $1$. Taking global sections on both sides, we obtain the statement.
\end{proof}

\begin{prop}\label{log purity}
    Let $S$ be a log regular fs log scheme and $U$ be a dense open subset of $S$.

    (1) The restriction functors
    \begin{align*}
        \mathcal{LLF}(S)&\to \mathcal{LLF}(U) \\
        (\mathrm{fin}/S)_{\mathrm{f}}&\to (\mathrm{fin}/U)_{\mathrm{f}}
    \end{align*}
    are faithful.

    (2) Suppose that $U$ contains all points on $S$ of codimension $1$. Then the restriction functors
    \begin{align*}
        \mathcal{LLF}(S)&\to \mathcal{LLF}(U) \\
        (\mathrm{fin}/S)_{\mathrm{f}}&\to (\mathrm{fin}/U)_{\mathrm{f}}
    \end{align*}
    are fully faithful.
\end{prop}

\begin{proof}
We shall prove only (2) because the same argument also works for (1). By Proposition \ref{log coordinate ring}, it is enough to prove the fully faithfulness of the former functor. Let $\cE_1, \cE_2$ be log vector bundles on $S$. It suffices to show that the map
    \[
    \Hom(\cE_1, \cE_2)\to \Hom(\cE_1|_U, \cE_2|_U)
    \]
    is bijective. By the limit argument (Proposition \ref{lim log loc free}), we may assume that $\mathring{S}$ is the spectrum of a strict local ring. Let $s$ be the unique closed point of $S$. Fix a chart $P\to \cM_S$ which is neat at $s$. 
    
    By Lemma \ref{log qcoh cl}, there exist an integer $n\geq 1$ such that $\cE_{1}|_{T}$ and $\cE_{2}|_{T}$ are classical, where we put $T\coloneqq S\times_{\bA_{P}} \bA_{P^{1/n}}$. Since $T$ is log regular by Lemma \ref{log reg stab}, it follows from Lemma \ref{log reg} (2) that $T$ is normal. Let $V$ be the preimage of $U$ by $T\to S$. Since $\mathring{T}\to \mathring{S}$ is a finite dominant morphism of normal schemes, the going down property holds for $\mathring{T}\to \mathring{S}$ by \cite[Proposition 00H8]{sp25}. Hence, $V$ contains all points on $T$ of codimension $1$, and so the restriction map
    \[
    \Hom(\cE_1|_T, \cE_2|_T)\to \Hom(\cE_1|_V, \cE_2|_V)
    \]
    is bijective. 

    Let $T':=T\times_S T$ and $V':=V\times_{U} V$. Note that $T'$ is not necessarily normal. Since there are isomorphisms
    \begin{align*}
        T'&\cong S\times_{\bA_{P}} \bA_{P^{1/n}}\times_{\bA_{P}} \bA_{P^{1/n}} \\ &\cong S\times_{\bA_{P}} \bA_{P^{1/n}}\times \bA_{P^{\mathrm{gp}}/nP^{\mathrm{gp}}} \\ 
        &\cong T\times \bA_{P^{\mathrm{gp}}/nP^{\mathrm{gp}}},
    \end{align*}
    natural projection morphisms $T'\to T$ are strict finite free. Hence, by Lemma \ref{bit general purity}, the restriction map
     \[
     \Hom(\cE_1|_{T'}, \cE_2|_{T'})\to \Hom(\cE_1|_{V'}, \cE_2|_{V'})
     \]
     is bijective. Then the statement follows from the following diagram in which both rows are exact:
     \[
     \begin{tikzcd}
         0 \ar[r] & \Hom(\cE_1, \cE_2) \ar[r] \ar[d] & \Hom(\cE_1|_T, \cE_2|_T) \ar[r] \ar[d,"\sim" sloped] & \Hom(\cE_1|_{T'}, \cE_2|_{T'}) \ar[d,"\sim" sloped]
     \\
     0 \ar[r] & \Hom(\cE_1|_U, \cE_2|_U) \ar[r] & \Hom(\cE_1|_V, \cE_2|_V) \ar[r] & \Hom(\cE_1|_{V'}, \cE_2|_{V'}).
     \end{tikzcd}
     \]
\end{proof}

The following proposition is a log version of the theorem of Tate and de Jong (\cite[Theorem 4]{tat67} and \cite[Corollary 1.2]{dj98}). This is proved in the case where $\mathring{S}$ is the spectrum of a discrete valuation ring by Bertapelle-Wang-Zhao in \cite[Theorem 5.19]{bwz24}. In general, we use the reduction to $1$-dimensional cases.

\begin{prop}\label{log Tate thm}
    Let $S$ be a log regular fs log scheme and $U$ be a dense open subset of $S$. Then the restriction functor
    \[
    (\mathrm{BT}/S)_{\mathrm{d}}\to (\mathrm{BT}/U)_{\mathrm{d}}
    \]
    is fully faithful.
\end{prop}

\begin{proof}
     Let $G$ and $H$ be log $p$-divisible groups over $S$. We shall prove the natural map
    \[
    \mathrm{Hom}(G,H)\to \mathrm{Hom}(G|_{U},H|_{U})
    \]
    is an isomorphism. The injectivity follows from Proposition \ref{log purity} (1). To show the surjectivity, we take a homomorphism $f_{U}\colon G|_{U}\to H|_{U}$. Fix an integer $n\geq 1$. By  \cite[Theorem 5.19]{bwz24} and Lemma \ref{log reg dvr}, there exists an open subset $V$ of $S$ containing $U$ and all points of codimension $1$ such that the homomorphism $f_{U,n}\colon G[p^{n}]|_{U}\to H[p^{n}]|_{U}$ induced from $f_{U}$ extends to a homomorphism $f_{V,n}\colon G[p^{n}]|_{V}\to H[p^{n}]|_{V}$. Furthermore, $f_{V,n}$ extends to a homomorphism $f_{n}\colon G[p^{n}]\to H[p^{n}]$ by Proposition \ref{log purity} (2). By Proposition \ref{log purity} (1), $f_{n}$ is the unique extension of $f_{U,n}$, and so, when $n$ varies, the system $\{f_{n}\}_{n\geq 1}$ gives a homomorphism $f\colon G\to H$ restricting to $f_{U}$.
\end{proof}

\begin{prop}\label{isog extends to isog}
    Let $S$ be a log regular fs log scheme and $U$ be a dense open subset of $S$. Let $f\colon G_{1}\to G_{2}$ be a homomorphism of log $p$-divisible groups over $S$. Suppose that $f|_{U}\colon G_{1}|_{U}\to G_{2}|_{U}$ is an isogeny. Then $f$ itself is an isogeny.
\end{prop}

\begin{proof}
    We may assume that $S$ is quasi-compact. Take an integer $d\geq 1$ such that $\mathrm{Ker}(f|_{U})$ is killed by $p^{d}$. Since $f|_{U}$ is surjective, there exists a unique homomorphism $g_{U}\colon G_{2}|_{U}\to G_{1}|_{U}$ such that $f|_{U}\circ g_{U}$ and $g_{U}\circ f|_{U}$ are the multiplication by $p^{d}$. By Proposition \ref{log Tate thm}, $g_{U}$ uniquely extends to a homomorphism $g\colon G_{2}\to G_{1}$, and both of $fg$ and $gf$ are the multiplication by $p^{d}$. In particular, $\mathrm{Ker}(f)$ is killed by $p^{d}$. Therefore, it follows from Proposition \ref{weaker def of isog} that $f$ is an isogeny.
\end{proof}

\section{Slope filtrations of log \texorpdfstring{$p$}--divisible groups}\label{sec 5}

The goal of this section is to prove Theorem \ref{ext of isog over log regular base}. We start with some technical lemmas.

\begin{lem}\label{fpqc descent of extendability}
    Let $f\colon Y\to X$ be a faithfully flat morphism of locally noetherian normal schemes. Let $U$ be a dense open subset of $X$ containing all points on $X$ of codimension $1$, and $V\coloneqq f^{-1}(U)$. Let $G_{U}$ be a finite locally free group scheme over $U$. Suppose that $G_{V}\coloneqq G_{U}|_{V}$ extends to a finite locally free group scheme $G_{Y}$ over $Y$. Then $G_{U}$ itself uniquely extends to a finite locally free group scheme over $X$.
\end{lem}

\begin{proof}
    Let $j\colon U\hookrightarrow X$ and $j'\colon V\hookrightarrow Y$ denote the natural open immersions. Let $f_{U}\colon V\to U$ be the restriction of $f$. For a finite locally free group scheme $H$ over a scheme $S$, we write $\cA_{H}$ for the  Hopf algebra object in $\mathcal{LF}(S)$ corresponding to $H$ (cf.~Proposition \ref{log coordinate ring}). By the flat base change theorem, there are isomorphisms 
    \begin{equation}
        f^{*}j_{*}\cA_{G_{U}}\cong j'_{*}(f_{U})^{*}\cA_{G_{U}}\cong j'_{*}\cA_{G_{V}}.
    \end{equation}
    Since $f$ has the going down property by the flatness of $f$, the open subset $V$ contains all points on $Y$ of codimension $1$. Hence, there are isomorphisms
    \begin{equation}
         j'_{*}\cA_{G_{V}}\cong j'_{*}{j'}^{*}\cA_{G_{Y}}\cong \cA_{G_{Y}},
    \end{equation}
    by the normality of $Y$. It follows from (5.1) and (5.2) that $f^{*}j_{*}\cA_{G_{U}}$ is a vector bundle on $Y$, and so $j_{*}\cA_{G_{U}}$ is also a vector bundle on $X$. Since the restriction functor $\mathcal{LF}(X)\to \mathcal{LF}(U)$ is fully faithful, the Hopf algebra structure on $\cA_{G_{U}}$ uniquely induces a Hopf algebra structure on $j_{*}\cA_{G_{U}}$, which corresponds to the desired finite locally free group scheme over $X$ extending $G_{U}$.
\end{proof}

\begin{lem}\label{purity of exactness}
    Let $X$ be a locally noetherian normal scheme and $U$ be a dense open subset of $X$ containing all points of codimension $1$. Let $G'\to G\to G''$ be a sequence of finite locally free group schemes over $X$. Suppose that $G''$ is finite \'{e}tale over $X$, and that the sequence
    \[
    0\to G'|_{U}\to G|_{U}\to G''|_{U}\to 0
    \]
    is exact. Then the sequence
    \[
    0\to G'\to G\to G''\to 0
    \]
    is also exact.
\end{lem}

\begin{proof}
    Since $G''$ is finite \'{e}tale over $X$, the morphism $G\to G''$ is finite flat. The image of $G\to G''$ is closed and contains a dense open subset $G''_{U}=G\times_{X} U$, and so $G\to G''$ is faithfully flat. What remains to be proved is that the composite $G'\to G\to G''$ is a zero map and that the induced homomorphism $G'\to \mathrm{Ker}(G\to G'')$ is an isomorphism. Both of them follow from the fully faithfulness of the restriction functor $(\mathrm{fin}/X)\to (\mathrm{fin}/U)$ (Lemma \ref{bit general purity}).
\end{proof}

\begin{prop}\label{purity of csd}
    Let $S$ be a log regular fs log scheme over $\bF_{p}$ with $\mathring{S}$ being regular, and let $U$ be a dense open subset of $S$ containing all points on $S$ of codimension $1$. Let $G_{U}$ be a completely slope divisible log $p$-divisible group over $U$. Then $G_{U}$ uniquely extends to a log $p$-divisible group over $S$.
\end{prop}

\begin{proof}
Since $G_{U}$ is completely slope divisible, there exists an exact sequence
    \[
    0\to G_{U}^{0}\to G_{U}\to G_{U}^{\et}\to 0
    \]
    as in Proposition \ref{dec of csd}. By \cite[Proposition 14]{zin01}, $G_{U}^{0}$ (resp.~$G_{U}^{\et}$) extends to a formal Lie group $H_{1}$ (resp.~an \'{e}tale $p$-divisible group $H_{2}$) over $\mathring{S}$.

    Note that $\cM_{S,\bar{s}}/\cO_{S, \bar{s}}^{\times}$ is a free monoid for all $s\in S$ by \cite[Chapter III, Theorem 1.11.6]{ogu18}. It is enough to prove the following two claims:
    \begin{enumerate}
        \item $G_{U}[p^{n}]$ uniquely extends to a log finite group scheme $G_{n}$ over $S$ for each $n\geq 1$;
        \item $G\coloneqq \varinjlim_{n\geq 1} G_{n}$ is a log $p$-divisible group over $S$.
    \end{enumerate}
    Hence, the limit argument (Proposition \ref{lim log fin grp sch}) allows us to assume that $\mathring{S}=\mathrm{Spec}R$ for a strict local ring $R$. Let $s\in \mathring{S}$ be the unique closed point. Take a chart $\bN^{r}\to \cM_{S,s}$ which is neat at $s$. Let $\alpha\colon \bN^{r}\to R$ denote the composite map $\bN^{r}\to \cM_{S,s}\to R$. By the log regularity of $S$, the sequence $(\alpha(e_{1}),\dots,\alpha(e_{r}))$ can be extended to a system of parameter of $R$, where $e_{i}$ is the $i$-th standard basis of $\bN^{r}$ for $1\leq i\leq r$. 
    
    We shall prove the claim (1) for a fixed $n\geq 1$. By Lemma
    \ref{log fin grp cl}, there exists an integer $m\geq 1$ such that $G_{U}[p^{n}]|_{V}$ is classical, where we set
    $T\coloneqq S\times_{\bA_{\bN^{r}}} \bA_{\frac{1}{m}\bN^{r}}$ and $V$ to be the preimage of $U$ by $T\to S$. Since the sequence $(\alpha(e_{1}),\cdots,\alpha(e_{r}))$ can be extended to a system of parameter of $R$, the scheme $\mathring{T}$ is regular.
    We have an exact sequence of finite locally free group schemes over $\mathring{V}$
    \begin{equation}
        0\to H_{1}[p^{n}]|_{V}\to G_{U}[p^{n}]|_{V}\to H_{2}[p^{n}]|_{V}\to 0.
    \end{equation}
    By the argument in \cite[page 7]{zin01} (or the fifth paragraph of the proof of \cite[Proposition 14]{zin01}), this sequence splits after taking the pullback by the $k$-th power of the Frobenius morphism $F_{\mathring{V}}^{k}\colon \mathring{V}\to \mathring{V}$ (which is a fpqc covering by the regularity of $\mathring{V}$) for some integer $k\geq 1$. In particular, the pullback of $G_{U}[p^{n}]|_{V}$ by $F_{\mathring{V}}^{k}$ extends to a finite locally free group scheme over $\mathring{T}$. Hence, by Lemma \ref{fpqc descent of extendability}, $G_{U}[p^{n}]|_{V}$ itself also extends to a finite locally free group scheme $G_{n,T}$ over $\mathring{T}$. By Lemma \ref{purity of exactness}, the exact sequence (5.3) uniquely extends to an exact sequence of finite locally free group schemes over $\mathring{T}$
    \begin{equation}
        0\to H_{1}[p^{n}]|_{T}\to G_{n,T}\to H_{2}[p^{n}]|_{T}\to 0.
    \end{equation}
    
    Since projection morphisms $T\times_{S} T\to T$ are strict finite free (see the proof of Proposition \ref{log purity}), the descent datum on $G_{U}[p^{n}]|_{V}$ over $V\times_{U} V$ uniquely extends to the descent datum on $G_{n,T}$ over $T\times_{S} T$ by Lemma \ref{bit general purity}. Therefore, $G_{n,T}$ descends to an object $G_{n}\in (\mathrm{fin}/S)_{\mathrm{f}}$ by Kummer log flat descent. Since the exact sequence (5.4) descends to an exact sequence
    \begin{equation}
        0\to H_{1}[p^{n}]\to G_{n}\to H_{2}[p^{n}]\to 0,
    \end{equation}
    $G_{n}$ is a log finite group scheme over $S$ by \cite[Proposition 2.3]{kat23}. This proves the claim (1).

    By construction, $G_{n}$ is killed by $p^{n}$ for any integer $n\geq 1$. Due to Proposition \ref{log purity} (2), the sequence $G_{U}[p]\hookrightarrow G_{U}[p^{2}]\hookrightarrow \dots$ uniquely extends to a sequence
    \[
    G_{1}\to G_{2}\to \dots.
    \]
    To prove the claim (2), it is enough to prove that the following claims are true:
    \begin{itemize}
        \item $G_{n}\to G_{n+1}$ is an injection for any integer $n\geq 1$;
        \item the sequence
        \[
        0\to G_{n}\to G_{n+m}\stackrel{\times p^{n}}{\to} G_{m}\to 0
        \]
        is exact.
    \end{itemize}
    These claims follow from the exact sequence (5.5) and the fact that similar properties hold for systems $\{H_{1}[p^{n}]\}$ and $\{H_{2}[p^{n}]\}$.
\end{proof}


\begin{thm} \label{ext of isog over log regular base}
    Let $S$ be a locally noetherian log regular fs log scheme over $\bF_p$ and $U$ be a dense open subscheme of $S$. Let $G$ be a log $p$-divisible group over $S$. Suppose that we are given a completely slope divisible log $p$-divisible group $H_{U}$ over $U$ and an isogeny $f_{U}\colon G|_{U}\to H_{U}$. Then there exist a completely slope divisible log $p$-divisible group $H$ over $S$ and an isogeny $f\colon G\to H$ extending $f_{U}$.
\end{thm}

\begin{proof}
    By the limit argument (Proposition \ref{lim log fin grp sch}), we may assume that the $\mathring{S}$ is the spectrum of a strict local ring with the unique closed point $s$. Fix a chart $P\to \cM_{S}$ which is neat at $s$ and an integer $d\geq 1$ with $\Ker(f_{U})\subset G[p^d]|_{U}$. By Lemma \ref{log fin grp cl}, there exists an integer $n\geq 1$ such that $G[p^d]|_{T}$ is classical, where we set $T\coloneqq S\times_{\bA_{P}} \bA_{P^{1/n}}$. Let $V$ be the preimage of $U$ by $T\to S$.

    Then, Nizio{\l}'s desingularization theorem \cite[Theorem 5.10]{niz06} gives a log regular fs log scheme $T'$ with $\mathring{T'}$ being regular and a proper birational morphism $\pi\colon T'\to T$. By Proposition \ref{ext of subgrp over log reg dvr} and the limit argument, there exist an open subset $V'$ containing $\pi^{-1}(V)$ and all points on $T'$ of codimension $1$ and an isogeny $f'_{V'}\colon G|_{V'}\to H_{V'}$ extending $f|_{\pi^{-1}(V)}$. The log $p$-divisible group $H_{V'}$ uniquely extends to a log $p$-divisible group $H_{T'}$ over $T'$ by Proposition \ref{purity of csd}, and the isogeny $f'_{V'}$ uniquely extends to an isogeny $f'\colon G|_{T'}\to H_{T'}$ by Proposition \ref{log Tate thm} and Proposition \ref{isog extends to isog}.
    
    We shall prove that $f'$ descends to $T$ by using the method of Oort-Zink; see \cite[Proposition 2.7]{oz02}. Let $\cM\to T$ be the moduli (non-log) scheme of isogenies from $G|_{T}$ whose kernel is contained in $G[p^d]|_{T}$. The isogeny $f'$ induces a morphism $\varphi\colon T'\to \cM$ of schemes. It suffices to prove that $\varphi$ factors as $T'\to T\to \cM$. Since $\cO_{T}\cong \pi_{*}\cO_{T'}$ by the normality of $T$, it is enough to show that, for $x\in T$, the morphism $\varphi$ maps $\pi^{-1}(x)$ to a single point. This follows from Lemma \ref{fin csd quot} and the connectivity of $\pi^{-1}(x)$.

    As a result, we obtain a log $p$-divisible group $H_{T}$ over $T$ with $(H_{T})|_{T'}\cong H_{T'}$ and an isogeny $f_{T}\colon G|_{T}\to H_{T}$ with $(f_{T})|_{T'}=f'$. Since $\pi$ is birational, $f_{T}$ is the extension of $(f_{U})|_{V}$ by Proposition \ref{log purity} (1). Hence, by Lemma \ref{bit general purity}, the two pullbacks of $\mathrm{Ker}(f_{T})$ by two projection morphisms $T\times_{S} T\to T$ are equal as subobjects of $G[p^{d}]|_{T\times_{S} T}$. Therefore, by Kummer log flat descent, $\mathrm{Ker}(f_{T})$ descends to a log finite subgroup of $G[p^{d}]$. Let $H\coloneqq G/\mathrm{Ker}(f_{T})$ and $f\colon G\to H$ be the natural surjection. Then the isogeny $f$ restricts to the given isogeny $f_{U}$. By Corollary \ref{generically csd implies csd}, $H$ is completely slope divisible. This proves the statement.
\end{proof}

\appendix

\section{Limit arguments for log schemes}\label{appendix}

In this appendix, we shall prove some fundamental results on limits of log schemes which are well-known to experts. We note that Lemma \ref{lim log str} for fine log schemes can be deduced easily from the fact that Olsson's stack $\mathcal{L}og_{S}$ is of finite presentation (\cite[Theorem 1.1]{ols03}). Here, we shall give a more direct proof. 

Let $\{ S_{i}\}_{i\in I}$ be a cofiltered system of coherent log schemes. In this section, we say that $\{ S_{i}\}$ satisfies the \textit{condition $(\ast)$} if the following conditions are satisfied.
\[
(\ast)
\begin{cases}
\text{(1) The log scheme $S_i$ is quasi-compact and quasi-separated for all $i\in I$.} \\
\text{(2) An arbitrary transition morphism $S_i\to S_j$ is affine and strict.}
\end{cases}
\]

When $\{ S_{i}\}$ satisfies the condition $(\ast)$, there exists a coherent log scheme $S$ which is the limit of $\{ S_{i}\}$ in the category of log schemes, and one can describe $S$ explicitly as follows. The underlying scheme of $S$ is the limit of the underlying schemes of $\{ S_{i}\}$. The log structure $\cM_{S}$ is the pullback log structure of $\cM_{S_{i}}$ by the natural projection morphism $S\to S_{i}$ for some $i$, which is independent of the choice of $i$ by the assumption $(\ast)$. Obviously, when $S_{i}$ is fine (resp.~fs) for each $i\in I$, the log scheme $S$ is the limit of $\{S_{i}\}$ in the category of fine (resp.~fs) log schemes.

\begin{lem} \label{lim gl sec}
    Let $\{ S_{i}\}_{i\in I}$ be a cofiltered system of coherent log schemes satisfying the condition $(\ast)$. We put $S\coloneqq  \varprojlim_{i\in I} S_i$.  Then the canonical morphism
    \[
    \varinjlim_{i\in I} \Gamma(S_i, \cM_{S_i})\to \Gamma(S, \cM_{S})
    \]
    is an isomorphism.
\end{lem}

\begin{proof}
    Let $p_{i}\colon S\to S_{i}$ be the projection morphism. We have the following diagram
    \[
     \begin{tikzcd}
     0 \ar[r] & \varinjlim p_{i}^{-1}\cO_{S_{i}}^{\times} \ar[r] \ar[d,"\sim" sloped] & \varinjlim p_{i}^{-1}\cM_{S_{i}} \ar[r] \ar[d] & \varinjlim p_{i}^{-1}(\cM_{S_{i}}/\cO_{S_{i}}^{\times}) \ar[d,"\sim" sloped]
     \\
     0 \ar[r] & \cO_{S}^{\times} \ar[r] & \cM_{S} \ar[r] & \cM_{S}/\cO_{S}^{\times} 
     \end{tikzcd}
     \]
     where the both rows are exact. Here, note that $p_{i}^{-1}\cM_{S_{i}}$ is not the pullback log structure but the inverse image of $\cM_{S_{i}}$ by $p_{i}$ as a sheaf. Hence, the canonical morphism $\varinjlim p_{i}^{-1}\cM_{S_{i}}\to \cM_{S}$ is an isomorphism. 

     For a quasi-compact and quasi-separated \'{e}tale morphism $U\to S$, we define a monoid $\cM'_{S}(U)$ as follows. We take $j\in I$, a quasi-compact and quasi-separated \'{e}tale morphism $U_{j}\to S_{j}$, and an isomorphism $U\cong U_{j}\times_{S_{j}} S$ over $S$. Here, we put $\cM'_{S}(U):=\displaystyle \varinjlim_{i\geq j} \cM_{S_{i}}(U_{j}\times_{S_{j}} S_{i})$. This definition is independent of the choice of $j$. Then $U\mapsto \cM'_{S}(U)$ defines a presheaf $\cM'_{S}$ on the category of quasi-compact and quasi-separated \'{e}tale morphisms to $S$. Since $\cM'_{S}$ satisfies the sheaf condition for quasi-compact and quasi-separated \'{e}tale coverings, $\cM'_{S}$ extends to an \'{e}tale sheaf on $S$ which is also denoted by $\cM'_{S}$. It follows from definition that $\cM'_{S}$ satisfies the same universal mapping property as $\varinjlim p_{i}^{-1}\cM_{S_{i}}$. Hence, the natural morphism $\varinjlim p_{i}^{-1}\cM_{S_{i}}\to \cM'_{S}$ is an isomorphism. The isomorphisms 
     \[
     \cM_{S}\cong \varinjlim p_{i}^{-1}\cM_{S_{i}}\cong \cM'_{S}
     \]
     give the desired isomorphism $\varinjlim \Gamma(S_i, \cM_{S_i})\cong \Gamma(S, \cM_{S})$
\end{proof}

\begin{lem} \label{lim log str}
    Let $\{ S_{i}\}_{i\in I}$ be a cofiltered system of quasi-compact and quasi-separated schemes in which every transition morphism is affine. We put $S\coloneqq \varprojlim_{i\in I} S_i$. 
    
(1) Suppose that the cofiltered category $I$ has a final object $0$. Let $\cM_0$, $\cN_0$ be log structures on $S_{0}$. Let $\cM_{i}$ (resp. $\cN_{i}$) denote the pullback log structure of $\cM_0$ (resp. $\cN_0$) by $S_{i}\to S_{0}$, and let $\cM$ (resp. $\cN$) denote the pullback log structure of $\cM_0$ (resp. $\cN_0$) by $S\to S_{0}$. Suppose that $\cM_{0}$ is coherent. Then the map
        \[
        \displaystyle \varinjlim_{i\in I} \Hom_{S_i}(\cM_{i}, \cN_{i}) \to \Hom_{S}(\cM, \cN) 
        \]
        is bijective, where $\Hom_{S_{i}}(\cM_{i}, \cN_{i})$ (resp. $\Hom_{S}(\cM, \cN)$) is the set of morphisms of log structures $\cM_{i}\to \cN_{i}$ (resp. $\cM\to \cN$).
        
(2) Let $\cM$ be a coherent (resp. fine) (resp. fs) log structure on $S$. Then, there exist $i\in I$ and a coherent (resp. fine) (resp. fs) log structure $\cM_i$ on $S_i$ such that $\cM$ is isomorphic to the pullback log structure of $\cM_i$.
\end{lem}

\begin{proof}
    (1) By working \'{e}tale locally on $S_{0}$, we may assume that there are a finitely generated monoid $P$ and a chart $P\to \cM_{0}$. Then the map in the statement factors as follows:
    \begin{align*}
    \displaystyle \varinjlim \Hom_{S_{i}}(\cM_{i}, \cN_{i})
    &\cong \varinjlim \Hom(P,\Gamma(S_{i},\cN_{i})) \\
    &\to \Hom(P,\varinjlim \Gamma(S_{i},\cN_{i})) \\
    &\to \Hom(P,\Gamma(S,\cN)) \\
    &\cong \Hom_{S}(\cM,\cN).
    \end{align*}
    Here, the second morphism is bijective by the fact that $P$ is finitely presented, and the third morphism is bijective by \ref{lim gl sec}. This proves (1).  

    (2) We prove only the assertion for coherent log structures because other assertion follows from the same argument. First, we consider the case where $\cM$ admits a chart $P\to \cM$ for a finitely generated monoid $P$. Since $P$ is finitely presented and $\varinjlim \Gamma(S_{i}, \cO_{S_{i}})$ is isomorphic to $\Gamma(S, \cO_{S})$, the morphism $P\to \cM\to \cO_{S}$ descends to a morphism $P\to \cO_{S_{i}}$ for some $i\in I$. Then the associated log structure to the morphism $P\to \cO_{S_{i}}$ is a desired one.
    
    We consider the general case. Take a quasi-compact and quasi-separated \'{e}tale covering $\pi\colon U\to S$ such that $\cM_{U}\coloneqq \pi^{*}\cM$ admits a chart. For some $i\in I$, the morphism $\pi$ descends to an \'{e}tale covering $\pi_{i}\colon U_{i}\to S_{i}$ and $\cM_{U}$ descends to a coherent log structure $\cM_{U_{i}}$ on $U_{i}$. Let $p_{k}\colon U\times_{S} U\to U$ and $p_{i,k}\colon U_{i}\times_{S_{i}} U_{i}\to U_{i}$ be natural projection morphisms for $k=1,2$. Due to the assertion (1), we may assume that the isomorphism $p_{1}^{*}\cM_{U}\cong p_{2}^{*}\cM_{U}$ descends to an isomorphism $p_{i,1}^{*}\cM_{U_{i}}\cong p_{i,2}^{*}\cM_{U_{i}}$ satisfying the cocycle condition over $U_{i}\times_{S_{i}} U_{i}\times_{S_{i}} U_{i}$ by replacing $i$ with a bigger one. This descent datum gives a coherent log structure $\cM_{i}$ on $S_{i}$ such that the pullback log structure of $\cM_{i}$ to $S$ is isomorphic to $\cM$. 
\end{proof}

\begin{prop}\label{lim log sch}
    Let $\{ S_{i}\}_{i\in I}$ be a cofiltered system of coherent log schemes satisfying the condition $(\ast)$. We put $S\coloneqq \varprojlim_{i\in I} S_i$. 
 
(1) Suppose that the cofiltered category $I$ has a final object $0$. Let $X_0$ and $Y_0$ be coherent log schemes over $S_{0}$. Suppose that $Y_{0}$ is of finite presentation over $S_{0}$. Then the natural map
        \[
        \varinjlim_{i\in I} \mathrm{Mor}_{S_{i}}(X_{0}\times_{S_{0}} S_i, Y_{0}\times_{S_{0}} S_i)\to \mathrm{Mor}_{S}(X_{0}\times_{S_{0}} S, Y_{0}\times_{S_{0}} S) 
        \]
        is bijective.
        
(2) Let $X$ be a coherent (resp. fine) (resp. fs) log scheme of finite presentation over $S$. Then there exist $i\in I$ and a coherent (resp. fine) (resp. fs) log scheme $X_i$ of finite presentation over $S_i$ with an isomorphism $X\cong X_{i}\times_{S_{i}} S$.
\end{prop}

\begin{proof}
    These statement are direct consequences of the analogous statements for classical schemes and Lemma \ref{lim log str}.
\end{proof}

\begin{lem}\label{lim chart}
    Let $\{ S_{i}\}_{i\in I}$ be a cofiltered system of coherent log schemes satisfying the condition $(\ast)$. We put $S\coloneqq \varprojlim_{i\in I} S_i$. Suppose that $S$ admits a chart $P\to \cM_{S}$ with a finitely generated monoid $P$. Then this chart comes from a chart $P\to \cM_{S_{i}}$ for some $i\in I$.
\end{lem}

\begin{proof}
Since $P$ is finitely presented, the given chart descends to a homomorphism $P\to \cM_{S_{i}}$ for some $i\in I$. Let $\cN_{i}$ be the log structure associated to the homomorphism $P\to \cM_{S_{i}}\to \cO_{S_{i}}$ and $\phi_{i}\colon \cN_{i}\to \cM_{S_{i}}$ be the morphism of coherent log structures induced from $P\to \cM_{S_{i}}$. Since the pullback of $\phi_{i}$ to $S$ is an isomorphism, the pullback of $\phi_{i}$ to $S_{j}$ is also an isomorphism for some $j\geq i$ by Lemma \ref{lim log str}(1). Then the homomorphism $P\to \cM_{S_{j}}$ induced from $P\to \cM_{S_{i}}$ is the desired chart.
\end{proof}

\begin{prop}\label{lim log properties}
    Let $\{ S_{i}\}_{i\in I}$ be a cofiltered system of fs log schemes satisfying the condition ($\ast$). We put $S\coloneqq \varprojlim_{i\in I} S_i$. Suppose that the cofiltered category $I$ has a final object $0$. Let $X_{0}$ be an fs log scheme of finite presentation over $S_{0}$. We put $X_{i}:=X_{0}\times_{S_{0}} S_{i}$ and $X:=X_{0}\times_{S_{0}} S$. Let $f_{i}\colon X_{i}\to S_{i}$ and $f\colon X\to S$ denote the natural morphisms.

    Let $\cP$ be one of the following properties of morphisms of fs log schemes. Then, if $f$ satisfies the property $\cP$, $f_{i}$ also satisfies the property $\cP$ for some $i\in I$. 

    \begin{enumerate}
        \item log flat
        \item log smooth
        \item log \'{e}tale
        \item Kummer log flat
    \end{enumerate}
\end{prop}

\begin{proof}
    These properties have a criterion using a chart of a morphism. For (1), this is just the definition. For (2) and (3), see \cite[Theorem 3.5]{kat89}. For (4), see \cite[Proposition 1.3]{int13}. Then the statement can be deduced from these criteria and Lemma \ref{lim chart}.
\end{proof}

\subsection*{Acknowledgments}
The author would like to thank his advisor, Tetsushi  Ito, for useful discussions and warm encouragement. He is also grateful to Kazuya Kato for useful discussions, and to the referee for carefully reading the manuscript and pointing out errors.
This work was supported by JSPS KAKENHI Grant Number 23KJ1325 and Graduate School of Science, Kyoto University under Ginpu Fund.

\end{document}